\documentclass{amsart}
\usepackage{times}
\usepackage{bbm}
\usepackage{amsfonts}
\usepackage{amsmath}
\usepackage{amscd}
\usepackage{amssymb}
\usepackage{latexsym}
\usepackage{amsthm}
\usepackage{stmaryrd}
\theoremstyle{plain}

\newtheorem*{conjectuur*}{Conjecture}

\swapnumbers
\newtheorem{theorem}[subsection]{Theorem}
\newcommand\Thm[1]{Theorem~\ref{#1}}
\newtheorem{corollary}[subsection]{Corollary}
\newcommand\Cor[1]{Corollary~\ref{#1}}
\newtheorem{lemma}[subsection]{Lemma}
\newcommand\Lem[1]{Lemma~\ref{#1}}
\newtheorem{proposition}[subsection]{Proposition}
\newcommand\Prop[1]{Proposition~\ref{#1}}

\newtheorem*{maintheorem}{Main Theorem}

\theoremstyle{definition}
\newtheorem{definition}[subsection]{Definition}

\newtheorem{example}[subsection]{Example}
\newcommand\Examp[1]{Example~\ref{#1}}

\theoremstyle{remark}

\newtheorem{remark}[subsection]{Remark}
\newcommand\Rem[1]{Remark~\ref{#1}}

\swapnumbers

\newcommand{\emptyprop}{q}

\newcommand \acf{algebraically closed field}
\newcommand \after{\circ}

\newcommand\ch{characteristic}
\newcommand \CM{Coh\-en-Mac\-au\-lay}
 
\newcommand \DVR{discrete valuation ring}

\renewcommand \hom [3]{\operatorname{Hom}_{#1}(#2,#3)} 
\newcommand \homo{homomorphism}

\renewcommand\iff{if and only if}%took out commas July 2005
\newcommand\into{\hookrightarrow}
\newcommand \inv[1]{{#1^{-1}}}
\newcommand \inverse[2]{{#1^{-1}(#2)}}
\newcommand \iso{\cong}

\newcommand \los{\L os' Theorem}
\newcommand \map[1]{{\newcommand{\tmpprop}{#1q}  \if\tmpprop\emptyprop \to\else \xrightarrow{{\phantom{i}{#1}\phantom{i}}}\fi}} 
\newcommand \maxim{\mathfrak m}
\newcommand \nat{\mathbb N}

\newcommand \op\operatorname
\newcommand \pol[2]{#1[#2]}

\newcommand \restrict [2]{\left.#1\right|_{{#2}}}
\newcommand \rij[2]{(#1_1,\dots,#1_{#2})}
\let\sub\subseteq

\newcommand\ul[1]{#1_\natural}
\newcommand\cp[1]{#1_\sharp}

\newcommand\asa{\leftrightarrow}
\newcommand\dan{\to}

\newcommand\niet{\neg}

\usepackage[bookmarks,colorlinks,pagebackref]{hyperref} %other options:
\input{xy}
\xyoption{all}

 \theoremstyle{plain}
\newtheorem*{thm}{Theorem}

\title {Universal Categories}
\author{Hans Schoutens}
\date\today%{11.11.04}
\address{Department of Mathematics\\
NYC College of Technology\\
City University of New York\\
NY, NY 11201 (USA)}
\email{hschoutens@citytech.cuny.edu}
\begin{document}

\maketitle

\begin{abstract}
The category of models of any  theory $T$ in any first-order language $L$ has the surprising property that any small category that is elementarily equivalent with it, already embeds in it. The proof uses an abstract argument via ultrapowers, leaving one wonder which concrete categorical axioms, depending on $T$ and $L$, are responsible for this embedding result.

We also propose a first-order logic for which  equivalent categories are always elementarily equivalent.
\end{abstract}

% start of the document
%:NEW COMMANDS 
\newcommand\formula[2]{\begin{description}
\item[$#1$ ] 
$#2$
\end{description}
}
\newcommand\fo{\mathbf{FO}}
\newcommand\qfo{\L o\'s}
\newcommand\cateo[1]{\textbf{cat}(#1)}
\newcommand\langcat{L_{\text{cat}}}
\newcommand\langfunc{\texttt{fct}(L)}

%{{\mathstrut^{\text{fct}}L}}
\newcommand\End[1]{\sym{End}(#1)}
\newcommand\sym[1]{{\tt #1}}
\newcommand\thcat{\mathbf{CAT}}
\newcommand\thfunc{\mathbf{Th}_{\text{func}}}
\newcommand\thcatpadd{\mathbf{Th}^+_{\text{cat}}}
\newcommand\thfuncpadd{\mathbf{Th}^+_{\text{func}}}
\newcommand\een{\sym{Id}}
\newcommand\mor[2]{\sym{Mor}(#1,#2)}
\newcommand\ic[1]{IsoCor(#1)}
\newcommand\lang{\mathcal L}
\newcommand\up[1]{\textbf{up}( #1)}
\newcommand \md[1]{\mathbb{M}\!\op{od}_\catset({#1})}
\newcommand \mds[1]{\mathbb{M}\!\op{od}^{\text{str}}_\catset({#1})}
\newcommand\cat[1]{\mathbbm {#1}}
\newcommand\set[1]{\left| #1\right |}

\newcommand\conc{$\catset$-concrete}
\newcommand\catset{\mathfrak U}

\section{Introduction}
The key notion of this paper is that of universality: let us call an $L$-structure   \emph{universal}, if   any elementarily equivalent $L$-structure of lesser cardinality embeds in it. Uncountably categorical structures are examples of universal structures, but beyond that, there seem to be no immediate examples. For instance, the real field is not universal, for their exist countable non-Archimedean real closed fields,
%\footnote{For instance, a countable elementary substructure of an  ultrapower of the reals.} 
which therefore are elementary equivalent with the reals but can never be embedded in them. 

In this paper, we will encounter nonetheless an abundance of examples, actually, one for each first-order theory $T$, to wit, the category  of models of $T$, viewed as a first-order structure in the language of categories. Since categories make a distinction between objects and morphisms, a two-sorted language $\langcat$ is therefore best suited for the model-theoretic study of categories. Apart from that, one needs a (partial) binary operation  for composing morphisms, a unary function assigning to an object its identity morphism, and two unary functions assigning to a morphism its domain and range. The axioms in this language are self-evident  and the resulting $\langcat$-theory $\thcat$ will be called the \emph{theory of categories}. To see how natural all this is: any $\langcat$-structure is in an obvious way a category, any substructure is a subcategory and the $\langcat$-\homo{s} between categories are just the functors. Moreover, this language is very expressive: the existence of any finite limit or co-limit is now first-order expressible. But the main observation is universality (see \Thm{T:univcat}):

\begin{maintheorem}\label{T:main}
Any category of the form $\md T$, for some theory $T$ in some first-order language $L$, is universal.
\end{maintheorem} 

Here we write $\md T$ for the subcategory of models of $T$ which belong to a once-and-for-all fixed Grothendieck universe $\catset$ (see \S\ref{s:GU} for more details). We need to make such an assumption to ensure that the category is small (i.e., is not a class), and therefore a model of $\thcat$. 

In this introductory section, let me sketch a proof in one particular case, that of the category of Abelian groups $\md{T_{\text{Ab}}}$, where $T_{\text{Ab}}$ is the theory of commutative groups in the language $L=(+,-,0)$. This argument is how one would like to prove the main theorem, by explicitly giving the axioms. Unfortunately, so far, this is basically the only case where I succeeded in doing so. The general  case will be proven using ultrapowers: it turns out that any ultrapower $\ul{\md T}$ (with respect to an ultrafilter belonging to the universe $\catset$) of any model category $\md T$ is already a subcategory of the original $\md T$. Together with a characterization of elementarity using ultrapowers, (\Prop{P:KSirr}, a corollary of the famous Keisler-Shelah criterion, but which can easily be proved directly), this will yield the main theorem. But let me explain what I consider to be a better proof at the hand of the category of Abelian groups.

\subsection{Abelian groups}\label{s:grp}

Let $\cat C$ be an arbitrary category in our Grothendieck universe $\catset$. Assume it admits a null object $0$, meaning  that for any  object $G$, there are unique morphisms $0\to G$ and $G\to 0$;  their composition $G\to 0\to G$ will be denoted again by $0$. We will also assume that $\mathbb C$ admits (finite)  products $G_1\times G_2$ with projection maps $\pi_1$ and $\pi_2$. There are therefore  unique morphisms $\Delta_G\colon G\to G\times G$ and  $\tau_G\colon G\times G\to G \times G$, the \emph{diagonal}  and the \emph{twist}  respectively, such that $\pi_k\after \Delta_G=1_G$  and $p_k\after\tau_G\after \Delta_G=p_l$, with $\{k,l\}=\{1,2\}$.\footnote{If the object $G$ happens to be a set, then $\Delta_G(a)=(a,a)$ and $\tau_G(a,b)=(b,a)$.} Let us call a morphism $\mu\colon G\times G\to G$ a \emph{group arrow}, if it  satisfies the following four conditions   
\begin{enumerate}
 \item \label{i:ass} (Associativity) $\mu\after(1_G\times \mu)=\mu\after (\mu\times 1_G)$;
\item\label{i:unit} (Unit) $\mu\after(0\times 1_G)=p_2$;
\item\label{i:inv} (Inverse) there exists a morphism $\lambda\colon G\to G$, such that $\mu\after (\een\times\lambda)\after\Delta_G=0$;
\item\label{i:comm} (Commutativity) $\mu\after\tau_G=\mu$;
\end{enumerate}
Tacitly assumed in these axioms are the various isomorphisms between products, such as $(G\times G)\times G\iso G\times (G\times G)$, which we simply denote by $G\times G\times G$, etc.
If $G$ is an actual (additive) Abelian group, then we can take $\mu(a,b):=a+b$  and $\lambda(a):=-a$. Note that commutativity is needed for $\mu$ and $\lambda$ to be actual group \homo{s}. 

To define the $\langcat$-theory $\mathbf {AB}$, we now add to the existence of finite products and a null-object, the two axioms stating (Ab1) that any object $G$ admits a unique group arrow $\mu_G$, and (Ab2), if $f\colon G\to H$ is an arbitrary arrow, with $\mu_G$ and $\mu_H$ the respective unique group arrows on $G$ and $H$, then 
\begin{enumerate}
\setcounter{enumi}4
\item\label{i:lin} (Linearity) $\mu_H\after (f\times f)=f\after\mu_G$.
%\item[Compatibility] $\mu_{G\times G}=\mu_G\times\mu_G$.
\end{enumerate}

To see that $\md{T_{\text{Ab}}}$ satisfies these axioms, taking the group arrow to define the group law on a structure, everything trivially holds except the uniqueness requirement in (Ab1). But this follows from   the following simple observation:

\begin{lemma}\label{L:unAb}
Let $(G,+,0)$ be an Abelian group and suppose $\mu\colon G\times G\to G$ defines a monoidal operation on $G$ with the same neutral element $0$. If $\mu$ is additive (for the induced addition on $G\times G$), then $\mu$ is just the group law.
\end{lemma}
\begin{proof}
By assumption, $\mu(a+a',b+b')=\mu(a,a')+\mu(b,b')$, for all $a,a',b,b'\in G$. Take $a'=0=b$, so that $\mu(a,b')=\mu(a,0)+\mu(0,b')$ and since $0$ is also the neutral element for the $\mu$-operation, the latter is just $a+b'$.
\end{proof}

Objects in a category $\cat C\models \mathbf{AB}$ are actually not yet groups, as the group arrow is only formal. But any  functor (preserving products and null objects) into the category of sets will turn them into actual groups and morphisms into group \homo{s}. This is best achieved by positing the existence of a generator $Z$ and taking as actualization the (faithful) functor $\hom{\cat C}Z-$ (for details, see  \S\ref{s:GU}). In this way, $\cat C$ becomes a subcategory of $\md{T_{\text{Ab}}}$. We will therefore also add the axiom asserting the existence of a generator to $\mathbf{AB}$, and so we proved:

%But it is now also clear that for any category $\mathbb C$ satisfying $\mathbf T_{\text{Ab}}$, every object $G$ can be endowed with the structure of an Abelian  group\footnote{One also needs to fix some forgetful functor in order to properly carries this out; we will discuss this in  } using the unique group arrow $\mu_G$, and any morphism in this category is a group \homo.  In particular, if we let $\mathbb A\text{b}_\catset$ be the category of all Abelian groups whose underlying set belongs to $\catset$,  then any category in $\catset$ which is elementary equivalent to  $\mathbb A\text{b}_\catset$ is just a subcategory of it, that is to say, we proved

\begin{corollary}\label{C:Ab}
The category $\md{T_{\text{Ab}}}$  is universal. \qed
\end{corollary}

The use of category theory to analyze model-theory is of course not new: Makkai's work led to what is now called FOCL (=first-order categorical logic; see \cite{MakPar,MakRey});  or Lawvere's perspective---from whom I also borrow   the use of the language $\langcat$---leading to ETCC (=elementary theory of the category of categories; see \cite{LawFunct,LawQS}) and its variants, just to mention a few. We use these perspectives in the present paper only in a very simple-minded way, by combining them in order to study some properties of the    models of the theory of a model-category. Often, these works  have focused  on complete theories, so that morphisms are taken to be elementary embeddings,  while we will look at the more natural category\footnote{No ordinary group theorist would  ever consider working  in the category of groups with just elementary embeddings.}
of models in a first-order language where the morphisms are just the \homo{s}.  Of course, in going from a first-order theory in an arbitrary language to the first-order theory of its category of models, something must be lost in translation, for there are only continu\-um many (complete) theories of (small) categories,  while there are far more languages $L$ and first-order $L$-theories. This raises the question how much is lost in translation? And how essential is this loss? Let me illustrate this with a simple example. Consider, let's say, the  first-order theory of fields FLD in the language of rings $L_{\text{ring}}$ and the corresponding category of its small models $\md {FLD}$ (see \S\ref{s:GU} for notation). Let $K$ be a field of particular interest (of possibly large cardinality) and take its elementary diagram in the language $L_{\text{ring}}(K)$.  The category of models of this latter theory is actually the co-slice category $K\downarrow\md{FLD}$ and therefore definable (with parameter $K$) inside $\md{FLD}$. As there are more than continuum many choices for $K$, there are non-elementarily equivalent  fields $K_1$ and $K_2$  whose co-slice categories are elementary equivalent. How are $K_1$ and $K_2$ then `algebraically related'? Since they are both definable in $\md{FLD}$,   therein then may lay the answer? 

In an appendix, we propose an answer to a question often raised: to which extent are elementary equivalence and equivalence of categories the same. In the usual language $\langcat$ (see \S\ref{s:lang} below), they definitely are not, as the sentence stating that whenever $A\iso B$ then $A=B$ holds in the skeleton of a category but not necessarily in the category itself (while both categories are always equivalent). The problem arises in the use of equality. In the literature, many solutions have been proposed (see, for instance, \cite{Blanc,FreydEquiv} or \cite[\S3]{Shul}), but they all seem  to use some dependent type theory, which is not the run-of-the-mill logic  model-theorists would normally consider.  Although not used frequently,   first-order logic without equality, seems more acceptable. But then function symbols become problematic, and so in stead we need to introduce a ternary predicate to substitute for the binary relation of composition. One also has to deal properly with the `core' of the category; details are discussed   in the Appendix (\S\ref{s:app}),  leading to

\begin{thm}\label{T:equivequiv}
There is a first-order logic in which  equivalent categories are elementarily equivalent.  
\end{thm} 

I do not know whether the converse is also true. This first-order logic, nonetheless,  is   sufficiently   expressive to formulate categorical properties such as the the existence of (finite) limits and co-limits. This also allows us to formulate a `non-evil'\footnote{i.e., not breaking the principle of equivalence.} version of our Main Theorem.

\section{The model-theory of categories}\label{s:lang}
While many categories have class-many objects and/or morphisms, technically, they cannot be considered as first-order structures. We will introduce the proper framework to deal with this issue, namely, Grothendieck universes, in \S\ref{s:GU}. For now, we simply assume that our categories are small (i.e., have set-many objects and morphisms). Let $\langcat$ be a two-sorted language, with the $\sym o$-sort for objects and the $\sym m$-sort for morphisms, together with a binary function $\after$ on the $\sym m$-sort, two unary functions $\sym {dom}$ and $\sym{rng}$ from the  $\sym m$-sort to the $\sym o$-sort, and a unary function $\sym{Id}$ from the $\sym o$-sort to the $\sym m$-sort subject. The interpretation of these symbols in an arbitrary category $\mathbb C$ are then respectively the composition $g\after f$ of morphisms $g\colon B\to C$ and $f\colon A\to B$, the domain $A$ and range $B$ of a morphism $f\colon A\to B$, and the identity morphism $1_A$ on an object $A$. Note that we treat $\after$ as a partial function and we do not care what it does on non-composable morphisms. We take the convention that variables ranging over the $\sym o$-sort are capitalized and those over the $\sym m$-sort are lower-case, so that we do not need to specify this each time. 

\begin{definition}\label{D:thcat}
Let $\thcat$ be the theory given by the following three axioms
\begin{enumerate}
\item \label{i:domrngcomp}  $(\forall x,y)$, if $ \sym {rng}(x)=\sym {dom}(y)$, then $ \sym {dom}(x)=\sym {dom}(y\after x)$ and $\sym {rng}(y)=\sym {rng}(y\after x)$;
\item \label{i:domrngass}$(\forall x,y,z)$, if  $\sym {rng}(x)=\sym {dom}(y)$ and $\sym {rng}(y)=\sym {dom}(z)$, then  $(z\after y)\after x=z\after (y\after x)$;
\item \label{i:id}  $(\forall X,y)$,   if $\sym{dom}(y)=X$ (respectively, $\sym{rng}(y)=X$), then $y\after \sym{Id}(X)= y$ (respectively, $\sym{Id}(X)\after y=y$);
\end{enumerate}
\end{definition} 
To conform more to the usual notation, we will adopt the following notational conventions for arbitrary $\langcat$-structures. If $f$ is a morphism with $A=\sym{dom}(f)$ and $B=\sym{rng}(f)$, then we simply indicate this by writing $f\colon A\to B$. The identity morphism $\sym{Id}(A)$ on an object $A$ is denoted simply $1_A$. This is justified by the following fact.

\begin{lemma}\label{L:catstr}
If $\mathcal C$ is a model of $\thcat$, then one can associate a category $\mathbb C$ to it such that as $\langcat$-structures, they are isomorphic. Moreover, if $\mathcal D$ is a second model of $\thcat$ with associated category $\mathbb D$, then the $\langcat$-\homo{s} between $\mathcal C$ and $\mathcal D$ induce precisely the functors  between $\mathbb C$ and $\mathbb D$.
\end{lemma} 
\begin{proof}
Consider the category $\mathbb C$ by taking for objects, the elements of the $\sym o$-sort in $\mathcal C$, and for morphisms, the elements of the $\sym m$-sort, with composition of morphisms given by $\after$ and the identity morphism given by $\sym{Id}$. The axioms of $\thcat$ then guarantee that $\mathbb C$ is a category and it is not hard to see that as an $\langcat$-structure, it is isomorphic to $\mathcal C$.  By definition of $\langcat$-structure, the interpretation of  the symbols imply that  a \homo\ $\varphi\colon \mathcal C\to\mathcal D$, sends a morphism $f\colon A\to B$ to a morphism $\varphi(f)\colon \varphi(A)\to \varphi(B)$ and preserves composition, showing that $\varphi$ induces a functor $\mathbb C\to\mathbb D$, and conversely.
\end{proof} 

We will therefore henceforward identify $\mathcal C$ with its associated category $\mathbb C$. In line with common notation, we also let $\hom{\mathbb C}AB$ denote the collection of all $f\colon A\to B$, that is to say, of all elements $y$ of the $\sym m$-sort such that $\sym{dom}(y)=A$ and $\sym{rng}(y)=B$. We refer to $\hom {\mathbb C}AB$ as a \emph{hom-set}. Note that it is an $\langcat$-definable subset in $\mathbb C$ with parameters $A$ and $B$.

\subsection{Limits and co-limits}\label{s:prod}
A major tool in category theory is the notion of limit. As far as finite limits are concerned, they are easily seen to be definable. We quickly review their definition. Given a functor $J\colon \mathbbm d\to \mathbb C$ (here $\mathbbm d$ is a category merely used for indexing purposes), a \emph{$J$-cone}   is a pair $(C,\psi)$, where $C$ is an object in $\mathbb C$, and for each object $A$ in $\mathbbm d$, we have a $\mathbb C$-morphism $\psi_A\colon C\to J(A)$ such that for each arrow $f\colon A\to B$ in $\mathbbm d$ we have $J(f)\after \psi_A=\psi_B$.  A $J$-cone $(L,\varphi)$ is called a \emph{limit} of $J$, if it is universal among all $J$-cones, that is to say, given any other $J$-cone $(C,\psi)$, there is a unique morphism $u\colon C\to L$,  such that $\varphi_A\after u=\psi_A$, for all objects $A$ in $\mathbbm d$.

\begin{example}\label{E:equal}
Let $\mathbbm d$ be the category with two objects and two morphisms between them (apart, of course, from the identity morphism on each object); symbolically $\bullet\rightrightarrows\bullet$. A functor $J$ from $\cat d$ into some category then yields two  `parallel morphisms' $A\rightrightarrows B$ in this category and a cone $(C,\varphi,\psi)$ then consists of two morphisms $\varphi\colon C\to A$ and $\psi\colon C\to B$, such that composing $\varphi$ with either morphism $A\to B$ yields $\psi$. The universal cone is then called the \emph{equalizer} of the two morphisms $A\rightrightarrows B$. If $A$ and $B$ are sets, then the equalizer  $E$ is just the subset of all elements in $A$ that have the same image in $B$ under both morphisms. 

Note that this last example can be extended to any two $L$-\homo{s}, where $L$ is a first-order language. Indeed, since \homo\ must preserve the constant and symbol functions of $L$, as a set, $E$ is then closed under these symbols, and therefore is again an $L$-structure. 
\end{example} 

If $\mathbbm d$ is finite (meaning both the number of objects as well as morphisms is finite), then one can easily spell this out by means of a first-order sentence in the language $\langcat$. Let me just show this for the case of a product (here,   $\cat d$ is the discrete category with two objects, symbolically $\bullet\quad\bullet$). A functor from this category  into $\mathbb C$ is then equivalent with picking two objects $X$ and $Y$. Their \emph{product} consists of an object, often denoted $X\times Y$, and two morphisms $\pi_X\colon X\times Y\to X$ and $\pi_Y\colon X\times Y\to Y$, such that given any object $Q$ and morphisms $\alpha_X\colon Q\to X$ and $\alpha_Y\colon Q\to Y$, we have a unique morphism $u\colon Q\to X\times Y$, yielding a commutative diagram (of `cones')
\begin{equation}\label{eq:prod}
\xymatrix{
&Q\ar[d]^u\ar[ddl]_{\alpha_X}\ar[ddr]^{\alpha_Y}\\
&X\times Y\ar[dl]^{\pi_X}\ar[dr]_{\pi_Y}\\
X&&Y.
}
\end{equation} 
It is clear that there is therefore a first-order sentence $\sym{Prod}$ claiming the existence of a product for any two objects. 

There is also a dual notion, \emph{co-limit}, which this times is a universal  co-cone, meaning that the arrows $\psi_A$ go from $J(A)$ to the co-cone. For instance, the \emph{co-product}, is given by the dual diagram, that is to say, in which all arrows in \eqref{eq:prod} are reversed. 

\begin{corollary}\label{C:lim}
Let $L$ be a   first-order language  and $\cat{Mod}(L)$  the category of $L$-models. Then $\cat{Mod}(L)$   is complete (i.e., closed under all  limits). Moreover, $\cat{Mod}(L)$ admits co-equalizers and filtered co-limits and whenever the language does not contain constant symbols and all function symbols are unary, then it is also co-complete. Among these (co-)limits, the finite ones are then definable.
\end{corollary} 
\begin{proof}
It is known that if a category admits equalizers and  products, then it is closed under all  limits (see, for instance, \cite[\S V.2, Thm 1]{MS}). While not definable, arbitrary products do exist in $\cat{Mod}(L)$. For filtered co-limits (\emph{aka} direct limits), see \cite[Theorem 2.4.6]{Hod}. Since I could not find a good reference for co-equalizers, I provide a proof here: given two $L$-\homo{s} $f,f'\colon \mathcal A\rightrightarrows \mathcal B$, define a relation on the underlying set $B:=D(\mathcal B)$ by $x\sim x'$, for $x,x'\in B$, if there is some $a\in A:=D(\mathcal A)$ such that $x=f(a)$ and $x'=f'(a)$. Let $E\sub B\times B$ be the transitive closure of $\sim$, which therefore is an equivalence relation and let $C:=B/E$ be the set of $E$-equivalence classes $\beta:=[b]$, for $b\in B$. We make $C$ into an $L$-structure $\mathcal C$ as follows. For a constant symbol $\sym c$, we let $\sym c^{\mathcal C}$ just be the equivalence class of $\sym c^{\mathcal B}$. For an $n$-ary function symbol $\sym F$, define $\sym F^{\mathcal C}([b_1],\dots,[b_n])$ as the equivalence class of $\sym F^{\mathcal B}(b_1,\dots,b_n)$ and for an $n$-ary relation symbol $\sym R$, declare that $(\beta_1,\dots,\beta_n)$ lies in $\sym R^{\mathcal C}$ \iff\ there are $b_i\in\beta_i$ such that $(b_1,\dots,b_n)\in \sym R^{\mathcal B}$. Only the case of a function symbol requires checking that it is independent from the choice of $E$-representative. So let $b_i'$ be $E$-equivalent to $b_i$. By an easy induction on the transitive closure, we may already assume that $b_i\sim b_i'$, say $b_i=f(a_i)$ and $b_i'=f'(a_i)$. Let's write $b$, $b'$ and $a'$ for the these respective $n$-tuples. Since $f$ and $f'$ are \homo{s}, we get
\[
\sym F^{\mathcal B}(b)=\sym F^{\mathcal B}(f(a))=f(\sym F^{\mathcal A}(a)) \qquad\text{and}\qquad \sym F^{\mathcal B}(b')=\sym F^{\mathcal B}(f'(a))=f'(\sym F^{\mathcal A}(a))
\]
showing that $\sym F^{\mathcal B}(b)\sim \sym F^{\mathcal B}(b')$.   It is also clear that the canonical map $B\to C\colon b\mapsto [b]$ induces a \homo\ $p\colon \mathcal B\to\mathcal C$ and by design,  $p\after f=p\after f'$.  To verify the universal property of co-equalizers, let $\mathcal M$ be an arbitrary  structure and $q\colon \mathcal B\to \mathcal M$ a \homo\ such that $q\after f= q\after f'$. Define a map $u\colon C\to D(\mathcal M)$ by the rule $[b]\mapsto q(b)$. To show that this is independent from the choice of representative $b$, let $b'$ be $E$-equivalent to $b$ and again we may assume by induction that $b\sim b'$, say $b=f(a)$ and $b'=f'(a)$, for some $a\in A$. Hence $q(b)=q(f(a))=q(f'(a))=q(b')$, showing that $u$ is well-defined and it is easy to see that it is an $L$-\homo\ and $q=u\after p$.  It is the unique \homo\ satisfying the latter factorization since $p$ is surjective. 

Assume next that $L$ has no constant symbols and all function symbols are unary. Then we define the co-product of a collection of $L$-structures $\mathcal M_x$, for $x\in X$, as the structure whose underlying set $U$ is the disjoint union of all $D(\mathcal M_x)$ and we endow it with an $L$-structure $\mathcal U$ by parallelism. That is to say, for an $n$-ary relation symbol $\sym R$, an tuple $(a_1,\dots,a_n)\in C$ lies in $\sym R^{\mathcal U}$ \iff\ there is an $x\in X$ such that all $a_i\in D(\mathcal M_x)$ and $\rij an\in \sym R^{\mathcal M_x}$. If $\sym F$ is a unary function symbol, then we let $\sym F^{\mathcal U}$ send an element $a\in U$ to $\sym F^{\mathcal M_x}(a)$, where $x\in X$ is the unique index such that $a\in D(\mathcal M_x)$.  Note that arities other than  one pose the problem that `mixed' tuples are ambivalent as to which function should apply to them (this includes constant symbols) and therefore have to be excluded. We leave it to the reader to verify that $\mathcal U$ satisfies the universal property of co-products. Co-completeness follows once more in this case from the existence of co-products and co-equalizers.
\end{proof} 
\begin{remark}\label{R:lim}
Whenever a (co-)limit exists, it is actually computed as in the category of sets. More precisely, the underlying-set functor $D$ reflects and preserves (co-)limits, if they exist. In particular, $\cat{Mod}(L)$ is  a regular (and even exact) category.

Note that the morphism $p\colon \mathcal B\to \mathcal C$ of the co-equalizer will not be in general strong, so the same result fails for the category of $L$-structures with strong \homo{s} (see \S\ref{s:strhomo} below for this category).
\end{remark}

\section{Ultraproducts and elementarity}\label{s:up}
Recall that an ultrafilter $\mathcal U$ on a set $X$ is a collection of non-empty subsets closed under finite intersection and superset, such that any subset of $X$ or its complement belongs to $\mathcal U$. One thinks of subsets in $\mathcal U$ as big and the others are then small. Given any sequence of sets $S_x$ indexed by $x\in X$, let $S_\infty$ be their Cartesian product. On $S_\infty$, the ultrafilter now induces an equivalence relation: we say that two elements $(a_x)_x$ and $(b_x)_x$ in $S_\infty$ are equivalent, if almost all their entries are the same, by which we mean that the set of all $x\in X$ such that $a_x=b_x$ lies in $\mathcal U$ (i.e., is big). We call the resulting quotient space the \emph{ultraproduct} of the $S_x$ and will often denote  it as $\ul S$. If $L$ is a first-order language and all $S_x$ are $L$-structures, then $\ul S$ inherits an $L$-structure from the Cartesian product $S_\infty$. The key result, called \los, is that an $L$-sentence $\sigma$ holds in $\ul S$ \iff\ it holds in almost all $S_x$ (see, for instance, \cite[\S 4.2]{Roth}). 

If all the $S_x$ are the same set $S$, so that $S_\infty=S^X$, then we call $\ul S$ the \emph{ultrapower} of $S$. It follows from \los\ that $\ul S$ is elementary equivalent to $S$. In fact, there is a canonical map $S\to \ul S$ sending an element $s\in S$ to the equivalence class of the constant map $x\mapsto s$, called the \emph{diagonal embedding}, and one checks that this actually yields an elementary embedding. By a famous theorem of Keisler and Shelah, elementarity can be detected by the ultrapower construction: two $L$-structures are elementary equivalent \iff\ they have some isomorphic ultrapowers (possibly with respect to different ultrafilters). Its proof, however, seems to rely on the generalized continuum hypothesis, but for our purposes, the following immediate corollary of it can easily be proved directly.

In the sequel, we will denote $L$-structures, for some first-order language $L$, by script letters $\mathcal A,\mathcal M, \dots$, and we let $D(\mathcal A), D(\mathcal M) \dots$ denote their underlying subsets.

\begin{proposition}\label{P:KSirr}
Two  first-order structures are elementary equivalent \iff\ there exists an elementary embedding from one into an ultrapower of the other.
\end{proposition} 
 \begin{proof}
 One direction is clear, so assume $\mathcal A$ and $ \mathcal B$ are elementary equivalent $L$-structures with underlying sets $A:=D(\mathcal  A)$ and $B:=D(\mathcal  B)$. We want to show that there is an ultrapower of $\mathcal A$ containing $\mathcal B$ as an elementary submodel. Let $b$ be a tuple enumerating all elements of $B$ and let $\sym b$ be a set of new constant symbols of the same size. Let $L(\sym b)$ be the language obtained from $L$ by adding these new constant symbols. Hence the structure $(\mathcal B,b)$ is an $L(\sym b)$-structure by interpreting   $\sym b$ by $b$. 
 Consider the theory $\op{Th}_{L(\sym b)}(\mathcal B,b)$ (aka, the \emph{elementary diagram} of $\mathcal B$) in this language $L(\sym b)$ and let $X$ be the collection of all non-empty finite subsets of this theory.  Any $\Sigma\in X$ is logically equivalent with a sentence of the form $\varphi(\sym b)$ such that  $(\mathcal B,b)\models \varphi(b)$. Since $\mathcal B$  is  then  a model of $ (\exists x)\varphi(x)$,   so is  $\mathcal A$ by elementary equivalence. Therefore   we can find an $L(\sym b)$-expansion $\mathcal A^*_\Sigma:=(\mathcal A, c_\Sigma)$ which models $\Sigma$, where $c_\Sigma$ is a tuple of elements that are the interpretation in $\mathcal A^*_\Sigma$ of $\sym b$. For each   $\Sigma\in X$, let $\langle \Sigma\rangle$ be the set of all $\Psi\in X$ with $\Sigma\sub \Psi$. Since the collection of all $\langle \Sigma\rangle$ for $\Sigma\in X$,  satisfies the finite intersection property, it is contained in some ultrafilter $\mathcal U$ on $X$. Let $\ul{\mathcal A}^*$ be the ultraproduct  of the $\mathcal A^*_\Sigma$ with respect to this ultrafilter. Hence the $L$-reduct $\ul{\mathcal A}$ of $\ul{\mathcal A}^*$ is the ultrapower of $\mathcal A$ with respect to the same ultrafilter. Let $\sigma$ be a sentence in $\op{Th}_{L(\sym b)}(\mathcal B,b)$. For each  $\Sigma\in X$ containing $\sigma$, we have $\mathcal A^*_\Sigma\models\sigma$, and by construction of $\mathcal U$, this therefore holds for almost all $\Sigma$, so that   $\ul{\mathcal A}^*$ satisfies $\sigma$.  Hence $\ul{\mathcal A}$ is the reduct of a model of $\op{Th}_{L(\sym b)}(\mathcal B,b)$, so that by the elementary diagram lemma (see, for instance, \cite[Lemma 8.2.1]{Roth}), $\mathcal B$ is an elementary substructure of $\ul{\mathcal A}$. In fact, the map $\mathcal B\to \ul{\mathcal A}$ is given by sending $b$ to the interpretation of $\sym b$  in the $L(\sym b)$-structure $\ul{\mathcal A}^*$, and it follows easily that this is an elementary embedding.
\end{proof}  
 \begin{remark}\label{R:ulproof}
Note that $X$, the underlying set of the ultrafilter,  has size at most that of the power set of the language $L(\sym b)$.
\end{remark}

\section{Ultrapowers of categories}\label{s:GU}

Recall that a Grothendieck \emph{universe}\footnote{The existence of universes does not follow from ZFC alone; it requires some large cardinal axiom like the existence of inaccessibles.}    is a set $\catset$ with the following properties
\begin{enumerate}
\item\label{i:utrans}$\catset$ is transitive, meaning that if $y\in \catset$ and $x\in y$, then $x\in  \catset$;
%\item if $x$ and $y$ belong to  $\catset$, then  so  does $\{x,y\}$
\item\label{i:upower} if $y$ belongs to  $\catset$, then so does its power-set $\mathcal P(y)$;
\item\label{i:uunion} if $I\in \catset$ and $x_i\in \catset$ for all $i\in I$, then the union of all $x_i$ is also in $\catset$;
\item\label{i:uomega} $\omega\in \catset$.
\end{enumerate}
%\comment{I don't know why we need (2):  if $x,y\in\catset$, then so their power-sets by \ref{i:upower},\\
% whence also $\{x\}$ and $\{y\}$ by \ref{i:utrans}, \\
% and hence the union $\{x,y\}$ by \ref{i:uunion}\\
%  (assuming that $2\in\catset$. But to have non-trivial examples, we should assume that $\omega\in\catset$).}

Fix, once and for all,  a universe $\catset$.   We view $\catset$ as a subcategory of the category of sets. The following two results show that a Grothendieck universe suffices to `do model-theory'.

\begin{lemma}\label{L:pullbackstr}
Let $\mathcal M$ be an $L$-structure, $X\sub\catset$  and  $s\colon X\to D(\mathcal M)$  a bijection. Then there is a unique $L$-structure $\mathcal X:=s^*\mathcal M$ such that $D(\mathcal X)=X$ and $s$ induces an isomorphism $\mathcal X\to \mathcal M$ of $L$-structures. 
\end{lemma} 
\begin{proof}
We define an $L$-structure on $X$ as follows. If $\sym R$ is a $n$-ary relation symbol (including the case $n=0$ which just means a constant symbol), then $\sym R^{\mathcal X}$ is the subset $\inverse s{\sym R^{\mathcal M}}$, and if $\sym f$ is an $n$-ary  function symbol, then $\sym f^{\mathcal X}$ is just $\inv s\after \sym f^{\mathcal M}\after s$. One easily checks that this yields the required $L$-structure.
\end{proof}

\begin{proposition}[L\"owenheim-Skolem for $\catset$]\label{P:LSGU}
Let $L$ be a first-order language in $\catset$,  let  $\mathcal M$ be an $L$-structure, and let $A\in\catset$ be a subset of $D(\mathcal M)$. Then there exists an elementary embedding $\mathcal N\preceq \mathcal M$  such that $A\sub D(\mathcal N)\in \catset$. 
\end{proposition} 
\begin{proof}
Let $\kappa$ be the cardinality of $\catset$, which therefore is bigger than those of $L$ and $A$. By the usual L\"owenheim-Skolem Theorem (see, for instance, \cite[Theorem 8.4.1]{Roth}), there exists an elementary substructure $\mathcal M'\sub\mathcal M$ of cardinality $\lambda<\kappa$ with $A\sub D(\mathcal M')$.  Choose some set $X\in\catset$ containing $A$ and having cardinality $\lambda$. By \Lem{L:pullbackstr}, there then exists an $L$-structure $\mathcal N\iso \mathcal M'$ with $D(\mathcal N)=X$.
\end{proof}

The model $\mathcal N$ thus constructed will be considered small:

\begin{definition}[Smallness]\label{D:small}
Given a first-order language $L\in\catset$, we call an $L$-structure $\mathcal N$ \emph{small}, if its underlying set $D(\mathcal N)$ is an element of $\catset$.
\end{definition} 
This applies in particular to categories, so that in our terminology, a category is small if both its collection of objects and its collection of morphisms  are members of $\catset$.  We will also encounter non-small categories, but even then we will assume that  the collection of objects as well as the collection of morphisms, form a subset of $\catset$ (hence they are still small in the classical sense). 
%Another common assumption is that any hom-set of a category is a member of $\catset$ (which then corresponds to the classical notion of locally small). A special case of this is given by the following definition

%%\footnote{Note that $Y^X$ belongs to $\catset$ if both $X$ and $Y$ do, and so the graph of any map $X\to Y$  also belongs to $\catset$.} 
%As far as (small) categories are concerned, they are two types we will consider: those that are a member of $\catset$, meaning that the collection of objects as well as the collection of morphisms both belong to $\catset$. A larger notion is that of a  \conc\ category:

\begin{definition}\label{D:concGU}
Let us say that a   category $\mathbb C$ is a \emph{$\catset$-category}, if all its hom-sets belong to $\catset$.\footnote{This then corresponds to the classical notion of locally small.} If $\cat C$ admits moreover a generator, then we will say that $\cat C$ is  
\emph\conc. We call a category \emph{sub-\conc}, if it is a subcategory of a \conc\ category.
\end{definition} 

Recall that an object $I$ in a category $\mathbb C$  is a \emph{generator} for that category if for each pair of distinct arrows $f\neq g\colon A\rightrightarrows B$, there is an arrow $a\colon I\to A$, such that $f\after a\neq g\after a$. 

\begin{corollary}\label{C:conccat}
Any (sub)-\conc\ category $\cat C$ embeds in $\catset$.
\end{corollary}
\begin{proof}
For $I$ a generator, the functor $\hom{\mathbb C} I-$ is by definition faithful (injective on morphisms); it is also injective on objects as different hom-sets are  disjoint.\footnote{The common convention in category theory.}
\end{proof}  

% Note that the collection of objects of $\mathbb C$, i.e., the elements of the $\sym o$-sort, is a member of $\catset$, 
%
%than also belongs to $\catset$ since  the (non-empty!) subsets $\hom{\mathbb C}IA$, for all objects $A$, form a partition of a subset of $\catset$ (since hom-sets can only be equal if they have the same domain and range).

A faithful representable functor can be thought of as some sort of forgetful functor, and that is precisely what we will do. Given a \conc\ category $\mathbb C$, we fix a generator $I$ and we set
\begin{equation}\label{eq:ff}
\set A:=\hom{\mathbb C}IA.
\end{equation} 
Note that generators are not unique (in fact, if $I$ is a generator, then so is any product $I\times A$), and therefore,  this forgetful functor depends on the choice of $I$, but we will not always make this explicit.  Given an arrow $f\colon A\to B$, the map $\set f\colon \set A\to \set B$ is given by composition, that is to say, given $a\in \set A$, so that $a\colon I\to A$, the image of $a$ under $\set f$ is then the morphism $f\after a\in \set B$.

Let $X$ be a set in $\catset$ and $\mathcal U$   an ultrafilter on   $X$ (so that also $\mathcal U\in \catset$);
%\footnote{This follows from the fact that any subset of a set in $\catset$ belongs to $\catset$ and an ultrafilter on $X\in \catset$ is just a subset of $\mathcal P\mathcal P(X)\in\catset$. Therefore, any ultraproduct with respect to $\mathcal U$ of a family of elements in $\catset$ is again in $\catset$.}
 we can therefore express this by  simply saying that $\mathcal U$ is an ultrafilter in $\catset$. Note that for a given category $\mathbb C$, viewing it as an $\langcat$-structure, we can then take its ultrapower $\ul{\mathbb C}$ and we have an embedding of categories $\mathbb C\into \ul{\mathbb C}$ via the diagonal embedding. Note that $\ul{\mathbb C}$ is the quotient of the product category $\mathbb C^X$ modulo the equivalence relation given by the ultrafilter $\mathcal U$ on $X$. Hence any object $B$ in the ultrapower is given by a map $x\mapsto B_x$, for $x\in X$ and $B_x$ an object in $\mathbb C$; we express this by saying that the map $x\mapsto B_x$ \emph{represents} $B$.

\begin{lemma}\label{L:isoulGU}
Let $\mathbb C$ be a \conc\ category with generator $I$. The image $\ul I$ of $I$ under the diagonal embedding is a generator for $\ul{\mathbb C}$, and so the latter is in particular again \conc, yielding a forgetful functor $\set-:=\hom{\ul{\mathbb C}}{\ul I}-$. 

Moreover,   for each object $B$ in $\ul{\mathbb C}$, the set $\set B$ is the ultraproduct of the sets $\set{B_x}$, where $x\mapsto B_x$ is a map representing the object $B$.
\end{lemma} 
\begin{proof}
The first statement follows from the second, since $\set-$ being faithful  on $\mathbb C$,  it remains so on $\ul{\mathbb C}$.  An element in $\set B=\hom{\ul{\mathbb C}}{\ul I}B$ is the equivalence class of a map  $x\mapsto f_x$, where each $f_x$ is a morphism $I\to B_x$, that is to say, an element in $\set{B_x}$, and  hence the map $x\mapsto \set{f_x}$ yields an element in the ultraproduct of the $\set{B_x}$, and conversely. It is now easy to see that this does not depend on any choices made.
\end{proof} 
\begin{remark}\label{R:isoulGU}
Nowhere in this proof did we use that the category $\mathbb C$ was  the same in the ultrapower, and so, the more general result is that if  $\mathbb C_x$ is a family of \conc\ categories with respective generator $I_x$, indexed by $x\in X$, and $\ul {\mathbb C}$ is their ultraproduct, then the object $\ul I$ of $\ul{\mathbb C}$ given by     the map $x\mapsto I_x$ is a generator for $\ul{\mathbb C}$  and for  each object $B$ in $\ul{\mathbb C}$, the set 
$\set B:=\hom{\ul{\mathbb C}}{\ul I}B$ is the ultraproduct of the sets $\set{B_x}$, where $x\mapsto B_x$ is a map representing $B$. 
\end{remark} 
\begin{remark}\label{R:subisoul}
The proof shows that we may weaken the assumption to $\mathbb C$ being sub-\conc.
\end{remark}

Note that $\catset$ is itself \conc, with generator the one-element set $\{\emptyset\}$. It follows from \Lem{L:isoulGU} that $\ul{\catset}$ is in fact a subcategory of $\catset$ since the ultraproducts also belong to $\catset$. We will see  in the next section that this `universality' is not an isolated phenomenon.

\begin{example}\label{E:freefunct}
If a category $\mathbb C$ admits a forgetful functor $\mathbb C\to \catset$ which has a left adjoint $U\colon \catset\to \mathbb C$, then $U(\{\emptyset\})$ is a generator for $\mathbb C$.   
%However, since this is only true up to isomorphism, to ensure that $\mathbb C$ is also \conc, we must require that each fiber of $\set-\colon \mathbb C\to \catset$ belongs to $\catset$.
\end{example} 

\begin{example}\label{E:grpinit}
Given a group (or, more generally, a monoid) $G$, we can associate a category $\mathbb G$ to it, with a single object $\bullet$  the morphisms of which are given by the group $G$, and their composition by the group law. The sole object is trivially a generator and the ``underlying set'' is $\set\bullet=\hom{}\bullet\bullet=G$, showing that the notion of forgetful functor is not just about objects. Clearly, $\ul{\mathbb G}$ is then representing the group (monoid) $\ul G$. 
\end{example} 

%\begin{example}\label{E:poinit}
%Let us call a category $\mathbb C$ (in $\catset$) \emph{PO}, if for any two objects, there is at most one morphism between them. If we declare $A\leq B$ whenever $\hom{\mathbb C}AB\neq\emptyset$, then this yields a partial order on the set of objects of $\mathbb C$, and conversely, any partially ordered set determines a PO category. Such a category has an initial object \iff\ the partial order has a least element $\bot$ and the underlying set $\set\bot$ is then just the collection of objects. Since PO is a first-order property, $\ul{\mathbb C}$ is again PO, and its underlying set is the ultrapower of that of $\mathbb C$.
%\end{example} 

\subsection*{Family of generators}
Given a subset of objects $\Sigma$ of a category $\mathbb C$, let us define, for any object $A$, the set (given  by the co-product/disjoint union in $\catset$)
\begin{equation}\label{eq:homfam}
\hom{\mathbb C}\Sigma A:=\coprod_{I\in\Sigma} \hom {\mathbb C} IA.
\end{equation} 
Note that $\Sigma$ is a family of generators (meaning that given distinct $f,g\colon A\rightrightarrows B$, there is at least one $I\in \Sigma$ and a morphism $a\colon I\to A$ such that $f\after a\neq g\after a$) \iff\ the functor  $\hom{\mathbb C}\Sigma-$ is faithful. We can therefore again use this to define a forgetful functor $\set-:=\hom{\mathbb C}\Sigma-$.  Let us call therefore a $\catset$-category $\mathbb C$ \emph{weakly \conc}, if it admits a family of generators.

We can now also prove a version of \Lem{L:isoulGU} for a weakly \conc\ category $\mathbb C$; we leave the proof to the reader.

\begin{lemma}\label{L:isoulfam}
Given a family $\Sigma$ of generators in a weakly \conc\ category  $\mathbb C$ and putting $\set-:=\hom{\mathbb C}\Sigma-$, let $\ul\Sigma$ be  all equivalence classes in $\ul{\mathbb C}$ given by maps $x\mapsto I_x$, with $I_x\in\Sigma$. Then $\ul\Sigma$ is a family of generators for $\ul{\mathbb C}$ and for each object $B$ in $\ul{\mathbb C}$ given by a map $x\mapsto B_x$, the set $\set B:=\hom{\ul{\mathbb C}}{\ul\Sigma}B$ is the ultraproduct of the sets $\set{B_x}$. \qed
\end{lemma} 
\begin{remark}\label{R:isoulfam}
Let us call the collection $\Sigma$ \emph{locally unique} if for each object $A$, there is a unique $I\in \Sigma$ such that $\hom{\mathbb C}IA$ is non-empty.  
It follows that $\ul\Sigma$ is then also locally unique.
\end{remark}

\section{The category of first-order models}\label{s:univ}

Let $L$ be a first-order language and $T$ an $L$-theory. We assume that $L$ is a member of $\catset$ (meaning that the collection of its constant, function and relation symbols is in $\catset$). Let $\md T$ be the category of small models of $T$ (that is to say all $L$-structures $\mathcal N\models T$  such that the underlying set $D(\mathcal N)\in\catset$), and  with  arbitrary $L$-\homo{s} for morphisms.  If $T$ is the empty theory, we simply write $\md L$, the \emph{category of $L$-structures}. 
%Given a model $\mathcal M$ of $T$, we denote its underlying set simply by $D(\mathcal M)$. 
%This notation is in accordance with \eqref{eq:ff} by the following result:

\begin{proposition}\label{P:concretestrucGU}
Let $L$ be a first-order language. Then $\md L$ is \conc, and the term algebra $\mathcal T$ in a single variable $x$ is a generator. Moreover, the functors $D$ and 
 $\set -=\hom{\md L}{\mathcal T}-$ are naturally isomorphic.
\end{proposition} 
\begin{proof}
Recall that the \emph{term} algebra $\mathcal T$ consists precisely of all the $L$-terms in the variable $x$, where the constant and function symbols of the language are interpreted by themselves, while the predicate symbols all define  empty sets. If $f,g\colon\mathcal M\rightrightarrows\mathcal N$   are two distinct $L$-\homo, then there is some $a\in D({\mathcal M})$ such that $f(a)\neq g(a)$. There is a unique \homo\ $s_a\colon \mathcal T\to \mathcal M$ sending the variable $x$ to $a$ (and each term $t(x)$ to its evaluation $t(a)$), and it follows that $f\after s_a\neq g\after s_a$.  By the same argument,  $\hom{\md L}{\mathcal T}{\mathcal M}$ is naturally isomorphic to the underlying set $D(\mathcal M)$, by sending a \homo\ $s\colon \mathcal T\to \mathcal M$ to  $s(x)\in D(\mathcal M)$.\footnote{In fact, the general term-algebra construction is a left adjoint of the functor $\set-$.} 
\end{proof} 
%\begin{remark}[FIX]\label{R:}
%More details below. 
%\end{remark} 

%\begin{proposition}\label{P:concretestrucGU}
%If $L$ has no relations symbols (other than constants), then $\md L$ is \conc.
%\end{proposition} 
%\begin{proof}
%Condition (C2) holds since the oll $L\in \catset$. 
%Conditions  (C1)--(C3) are clear with respect to the forgetful functor $\set-$, so it remains to  construct its left adjoint $U\colon \catset \to \md L$. For each $X\in\catset$, let $U(X)$ be the \emph{term-algebra} $\op{Term}_L(X)$   on $X$,\footnote{In other words, the collection of all formal terms in these variables and the symbols from $L$; see, for instance, \cite[\S2.2]{Roth}.} where we consider the elements of $X$ as  free variables. It is trivially an $L$-structure  belonging to $\catset$. 
% We verify the adjoint property \eqref{eq:freefunctor}: given a morphism $f\colon U(X)\to \mathcal M$, for some $L$-structure $\mathcal M$, let $\set f$ be the map $X\to \set{\mathcal M}$ sending $x\in X$ to the image of   $x$  (viewed as a variable) in the term algebra $\op{Term}_L(\lambda)$. Conversely, given a map $g\colon X\to \set{\mathcal M}$, in order in $\op{Term}_L(\lambda)$, and  for this we take the image of $x$ to be $g(x)$.
%\end{proof} 

\begin{corollary}\label{C:univstrucGU}
Given an ultrafilter $\mathcal U$   on $X\in \catset$,   the ultrapower $\ul {\md L}$ with respect to $\mathcal U$ is isomorphic to a subcategory of $\md L$. The same is true upon replacing $\md L$ by $\md T$ where $T$ is any $L$-theory.

More precisely, there is an embedding of categories $i\colon \ul{\md L}\into \md L$, making the following diagram commute
\begin{equation}\label{eq:univstrucGU}
\xymatrix{
\ul{\md L}\ar[r]^i\ar[dr]_{{\set -}}&\md L\ar[d]^D\\
&\catset
}
\end{equation} 
where $\set M:=\hom{\ul{\md L}}{  { I}}M$.
%
%In fact, there exists an elementary embedding $\ul{\md T}\to \md T$. (false)
\end{corollary} 
\begin{proof}
Let $M$ be an object of  $\ul{\md L}$ given by a sequence $x\mapsto \mathcal M_x$, for $x\in X$, with each $\mathcal M_x$ an object in $\md L$, that is to say, an $L$-structure.  As in \Prop{P:concretestrucGU}, let $\mathcal T$ be the generator of $\md L$ given by the one-variable term algebra,   and let $  { I}$ be the object in $\ul{\md L}$ given by the constant sequence $x\mapsto \mathcal T$.  
%By \los, $ I$ is a generator for $\ul{\md L}$. 

%By \eqref{eq:objectpointGU}, we have a natural isomorphism between $\set M\in \catset$ and the hom-set $\hom{\ul{\md L}}{\ul { I}}M$.
%
% On the other hand, b
 By \Lem{L:isoulGU},  the set $\set M$ is equal to the ultraproduct of the sets $\set{\mathcal M_x}$. Let $\ul{\mathcal M}$ be the ultraproduct of the $\mathcal M_x$. By \Prop{P:concretestrucGU}, there are natural isomorphisms $\set{\mathcal M_x}\iso D(\mathcal M_x)$, and hence, by \los,  a natural isomorphism 
\begin{equation}\label{eq:natiso}
 \set M\iso D(\ul{\mathcal M}).
\end{equation}   
By \Lem{L:pullbackstr} therefore, we can define an $L$-structure  on the set $\set M$ by pulling back the $L$-structure of $\ul{\mathcal M}$ and we let $i(M)$ be this $L$-structure. In particular,  $i( M)\iso \ul{\mathcal M}$. To see that $i$ is functorial, let $N$ be a second object in $\ul{\md L}$ coming from a sequence of $L$-structures $x\mapsto \mathcal N_x$, and let $i(N)$ be the canonical $L$-structure obtained on $\set N$ isomorphic to the  ultraproduct  $\ul{\mathcal N}$ of  the ${\mathcal N_x}$. A morphism $f\colon M\to N$ in $\ul{\md L}$, arises from a sequence of morphisms $f_x\colon \mathcal M_x\to \mathcal N_x$ in $\md L$, which in the ultraproduct yields a morphism of $L$-structures $\ul f\colon \ul{\mathcal  M}\to\ul{ \mathcal N}$. Naturality of \eqref{eq:natiso} now yields a morphism $i(f)\colon i(M) \to i(N)$. Since $\set -$ is an embedding, so is therefore $i$ by \eqref{eq:univstrucGU}.
 
% Since hom-sets can only be equal if they have the same domain and co-domain, the functor $M\mapsto  \mathcal M$ is injective on objects as well as morphisms.

To prove the   statement for arbitrary theories, we cannot just copy the above proof since $\mathcal T$ might not be a model of $T$. However, we can invoke \Rem{R:subisoul}, to wit, with $M$ an object in $\ul{\md T}$,  the $\mathcal M_x$ are now also models of $T$, whence so is their ultraproduct $\ul{\mathcal M}$ by \los. As $i(M)\iso \ul{\mathcal M}$, the embedding $i$ then factors through $\md T$, as desired. 
%
%and the induced structure $\mathcal M$ on $\set M$ as in the above proof,   is again a model of $T$ by \los. So the embedding $j\colon \ul{\md L}\into \md L\colon M\mapsto \mathcal M$ restricts to an embedding $j\colon \ul{\md T}\into \md T$.
%
%To prove the last statement,  that the embedding $j\colon \ul{\md T}\into \md T$ is elementary, we use the Tarski-Vaught test: let $\varphi(x,y_1,\dots,y_n)$ be a quantifier free formula in $\langcat$ and let $b$ be an $n$-tuple of parameters from $\ul{\md T}$. We have to show that if there exists $a\in \md T$ such that $\md T\models \varphi(a,j(b))$, then we can already find $c\in \ul{\md T}$ such that $\md T\models \varphi(j(c),j(b))$. Choose a map $x\mapsto b_x$, with $b_x$ an $n$-tuple in $\md T$ representing the tuple $b$ in $\ul{\md T}$. 
%Via the diagonal embedding, $\varphi$ holds in $\ul{\md T}$, whence by \los, $\varphi(a,b^1_x,\dots,b^n_x)$ holds for almost $x$ in $\md T$
\end{proof} 
\begin{remark}\label{R:univstrucGU}
Note that embedding $i\colon\ul{\md L}\to \md L$ is not full, as the morphisms in the former are   coming from ultraproducts. Pre-composing with the diagonal embedding $\md L\into \ul{\md L}$, the ensuing endofunctor on $\md L$  sends, up to a natural isomorphism,  an $L$-structure $\mathcal M$ to its ultrapower $\ul{\mathcal M}$, and \homo{s} to their ultrapower. In general this is not an elementary map, whence neither is $i$. Indeed, even with $L=\emptyset$, i.e., for the category of sets, this does not hold: the sentence expressing that there is a bijection between $\omega$ and $2^\omega$ is clearly false downstairs, but it is true upstairs since $\ul{(2^\omega)}\iso\ul\omega$. In particular, the embedding $i$ might even fail to reflect isomorphisms.
%
%Moreover, the embedding is not reflexive either, since the diagonal embedding is not a left adjoint.
\end{remark} 

\subsection*{Strong \homo{s}}\label{s:strhomo}
The category  $\md L$ contains the  (non-full) subcategory   $\mds L$, with the same objects, the $L$-structures, but where the \homo{s} are now assumed to be strong, meaning  that for any given relation symbol $\sym R$, a tuple satisfies $\sym R$ \iff\ its image under the morphism satisfies $\sym R$. Of course, if the language has no relation symbols (like, for instance, $\langcat$), then $\md L=\mds L$. In general, let us write $\langfunc$ for the restriction of $L$ obtained by omitting all predicate symbols (of positive arity). 

\begin{lemma}\label{L:pullbackstrucGU}
Let $\mathcal N$ be an $\langfunc$-structure and $\mathcal M$  an $L$-structure. If  $f\colon \mathcal N\to \restrict{\mathcal M}\langfunc$ is an   $\langfunc$-\homo, then it induces an $L$-structure on $\mathcal N$, denoted $f^*\mathcal M$, such that $f\colon f^*\mathcal M\to \mathcal M$ is a strong $L$-\homo.  Moreover, $f^*\mathcal M$ is the only expansion of $\mathcal N$ for which $f$ becomes strong.
\end{lemma} 
\begin{proof}
Let $\sym R$ be an $n$-ary relation symbol and $P\sub D({\mathcal M})^n$ the subset it defines in $\mathcal M$. Then we let the interpretation of $\sym R$ on $D({\mathcal N})$   be the subset $\inverse{f}{P}$ (where we continue to write $f$ for the $n$-fold product $f\colon D({\mathcal N})^n\to D({\mathcal M})^n$). By construction, $f$ is then a strong $L$-\homo. The last assertion is now clear from the definition of strong \homo.
\end{proof} 

As before, let $\mathcal T$ be the one-variable term algebra. We cannot expand it into an $L$-structure which would work for all $L$-structures, so we instead consider the collection   $\Theta$   of all $L$-structures whose reduct to $\langfunc$ is $\mathcal T$ (clearly $\Theta\in\catset$). 

\begin{lemma}\label{L:genstrucGU}
The set $\Theta$ forms a family of generators of $\mds L$. This family is moreover locally unique.
\end{lemma} 
\begin{proof}
Let $f,g\colon\mathcal M\rightrightarrows \mathcal N$ be two distinct  (strong) \homo{s} of $L$-structures. Their  reducts are $\langfunc$-structures and so by \Prop{P:concretestrucGU}, there exists a  morphism $a\colon \mathcal T\to \restrict{\mathcal M}\langfunc$ such that $f\after a\neq a\after g$. With $\mathcal J:=a^*\mathcal M$---whence $\mathcal J\in \Theta$---, we get the desired  (strong) \homo\ $a\colon \mathcal J\to \mathcal M$ by \Lem{L:pullbackstrucGU}. Moreover, uniqueness now guarantees that $\hom{\mds L}{\mathcal J'}{\mathcal M}$ is empty, for any $\mathcal J'\in\Theta$ different from $\mathcal J=a^*\mathcal M$, showing that the family is locally unique.
\end{proof} 

%For the next result, given a subset $\mathbf G$ of objects in a category $\mathbb C$ and an object $H$, we define
%\[
%\hom{\mathbb C}{\mathbf G}H:=\coprod_{G\in \mathbf G}\hom{\mathbb C}GH.
%\]
%Recall that $\coprod$ denotes the disjoint union of sets.

\begin{proposition}\label{P:setstrucGU}
With $\Theta$ defined as above, there is a natural isomorphism of functors, such that for each small $L$-structure  $\mathcal M$, we have 
\begin{equation}\label{eq:setstruc}
D({\mathcal M})\iso \hom{\mds L}{\Theta}{\mathcal M}.
\end{equation}  
\end{proposition} 
\begin{proof}
Let $y\in D({\mathcal M})$. There is a unique $\langfunc$-\homo\ $a_y\colon \mathcal T\to \restrict{\mathcal M}\langfunc$ sending the free variable to $y$. By the proof of \Lem{L:genstrucGU}, the object $a_y^*\mathcal M$ is therefore the unique structure in $\Theta$ such that $a_y\colon a_y^*\mathcal M\to \mathcal M$ is strong.

Conversely, suppose $\mathcal J\in \Theta$ and we have a strong \homo\ $a\colon \mathcal J\to \mathcal M$.  The induced map $a$ from $D( {\mathcal J})=D({\mathcal T})$ to $D({\mathcal M})$ sends the free variable to an element $y\in D({\mathcal M})$. Since $\Theta$ is locally unique, we must have $a=a_y$, proving that this assignment is the inverse of the previous one.
\end{proof} 
%\begin{remark}\label{R:setstrucGU}
%We actually showed that the collection $\mathbf I$ is a  {strong} family of generators, where an arbitrary  family $\mathbf I$ of generators for a category $\mathbb C$ is called \emph{strong} if for each object $A$ of $\mathbb C$, there is exactly one $I\in \mathbf I$, such that $\hom{\mathbb C}IA$ is non-empty.
%\end{remark} 

\begin{corollary}\label{C:univstrucGUs}
Let $L$ be a first-order language and $\mathcal U$ an ultrafilter, both in $\catset$.  Given an $L$-theory $T$, the ultrapower $\ul {\mds T}$ with respect to $\mathcal U$ is isomorphic to a subcategory of $\mds T$. 
\end{corollary} 
\begin{proof}
We can copy the proof of \Cor{C:univstrucGU}, once we prove the equivalent statement of \Lem{L:isoulGU} in this setup. To this end, we may take $T$ to be the empty theory. Therefore, let $M$ be an object in $\ul{\mds L}$, given by a map $x\mapsto \mathcal M_x$, for $x\in X$, with each $\mathcal M_x$ an object in $\mds L$. Let $\ul{\Theta}$ be the image of the set $\Theta$ in $\ul{\mds L}$, that is to say,  all equivalence classes of maps  $X\to \Theta$. Define
\[
\set M:=\hom{\ul{\mds L}}{\ul{\Theta}}M.
\]
By \Lem{L:isoulfam}, the set  $\set M$, being an ultraproduct of $L$-structures, inherits a canonical  $L$-structure $\mathcal M$, and the rest of the proof now goes through as in \Cor{C:univstrucGU}.
%
% once we showed that it  is isomorphic to the ultraproduct $\ul M$ of the $\set {\mathcal M_x}$ (the analogue   of \Lem{L:isoulGU}). To this end, let $\ul y\in \ul M$, given by   a sequence   of elements $y_x\in\set{\mathcal M_x}$, for $x\in X$. By \Prop{P:setstrucGU}, each $y_x$ determines   a unique $L$-\homo\ $s_x\colon \mathcal I_x\to \mathcal M_x$, for a unique $\mathcal I_x\in \mathbf I$. In $\ul{\mds L}$, this yields a morphism $\ul I\to B$, where $\ul I$ is an object in $\ul{\mathbf I}$, and hence an element   $y\in \set B$. Since everything is with respect to the same ultrafilter, the map $\ul M\to \set B\colon \ul y\mapsto y$ is easily seen to be well-defined. To construct its inverse, let $y\in \set B$, which therefore is some morphism $\ul I\to B$ in $\ul{\mds L}$, for some $\ul I\in \ul{\mathbf I}$. Hence, it must come from a sequence of morphisms $y_x\colon \mathcal I_x\to \mathcal M_x$, with each $\mathcal I_x\in \mathbf I$. By \Prop{P:setstrucGU}, this yields for each $x$ a unique element $y_x\in \set{\mathcal M_x}$, and hence an element $\ul y\in \ul M$. Again the map $\set B\to \ul M\colon y\mapsto \ul y$ is well-defined, and  is easily checked that it is the inverse from the previous map. 
%Hence $\ic{\set B}$ and $\ul M$ are isomorphic, and since both are in $\catset$,  they must be equal by \Rem{R:isocorr}.
\end{proof}

\section{Universal categories}

Let $\mathcal M$ be an $L$-structure. We call $\mathcal M$ \emph{universal} (respectively, \emph{strongly universal}), if  every $L$-structure $\mathcal N$ which is elementarily equivalent with $\mathcal M$ and has cardinality less than $\mathcal N$, embeds (respectively, elementarily embeds) in $\mathcal M$.

Any countable structure is trivially universal according to this definition, so we exclude this case. 

\begin{example}\label{E:univACF}
Any uncountable \acf\ $K$ is strongly universal (in the language of rings). This is just a special case of:
%Indeed, if $F$ is elementarily equivalent with $K$,  and $\#(F)<\lambda:=\#(K)$, then $F$ is also algebraically closed and has the same \ch\ as $K$. Let $\xi$ be a transcendence base of size $\lambda$ over $F$, so that the algebraic closure of $F(\xi)$ is a field of the same \ch\ and cardinality as $K$, and hence by Steinitz's Theorem, the two fields are isomorphic. The induced embedding $F\into K$ is then elementary by Quantifier Elimination. 
\end{example} 

%Of course, the theory of \acf\ is uncountable categorical, and so, we expect this to be true under this more general assumption.  

\begin{proposition}\label{P:univcatth}
Let $T$ be a complete, uncountable categorical $L$-theory and $\mathcal M$  a model   of $T$  of uncountable cardinality $\lambda$. If $\lambda$ is at least the size of the language $L$, then $\mathcal M$   is strongly universal.
\end{proposition}
\begin{proof}
Let $\mathcal N$ be a model of cardinality   $\beta<\lambda$. By upward L\"owenheim-Skolem, there exists an elementary extension $\mathcal N'$ of $\mathcal N$ of size $\lambda$ (see, for instance,  \cite[Theorem 8.4.3]{Roth}), and by categoricity, the former must be isomorphic to $\mathcal M$.
\end{proof}  
%\begin{remark}\label{R:univcatth}
%We can replace the restriction given by the size of the language under some additional set-theoretic assumption, using instead  \Rem{R:ulproof}, to  yield  an ultrafilter on a set of cardinality at most $\beta$, such that $\mathcal N$ is an elementary sub-model of the ultrapower $\ul{\mathcal M}$ with respect to this ultrafilter. The cardinality of $\ul{\mathcal M}$ is   at most $\lambda^\beta$ and assuming that this is equal to $\lambda$, categoricity implies $\mathcal M\iso\ul{\mathcal M}$. The equality $\lambda=\lambda^\beta$ follows for instance from the Generalized Continuum Hypothesis (GCH): take   $\alpha$ such that $\lambda=2^\alpha$, and hence $\lambda^\beta=2^{\alpha\cdot \beta}=2^\alpha=\lambda$, since $\beta\leq \alpha$.
%\end{remark} 

%We continue to assume that $\kappa$ is an inaccessible cardinal.

Rather than working with cardinalities,  the following definition in terms of Grothendieck universes is more suitable for our needs. Recall that $D(\mathcal M)$ denotes the underlying set of an $L$-structure and \emph{small} means that this set belongs to $\catset$.

\begin{definition}\label{D:univ}
An $L$-structure $\mathcal M$ is called \emph{universal}, if $D(\mathcal M)\sub\catset$ and  for any small $L$-structure $\mathcal N$  elementarily equivalent with it, there is an embedding $\mathcal N\into \mathcal M$.
\end{definition} 

For cardinality reasons, the underlying subset of a universal structure can never be an element of $\catset$, only a subset. Universality implies that  small substructures form     an elementary class:

\begin{proposition}\label{P:univth}
Let $\mathcal M$  be a universal structure in a first-order language $L$,  let $T$ be its complete $L$-theory, and let  $\mathcal N$ be a small $L$-structure. Then  $\mathcal N$ is (isomorphic to) a substructure of $\mathcal M$ \iff\ it is a model of $T_\forall$. 
\end{proposition} 
\begin{proof}
One direction is clear by   universal preservation (see \cite[Lemma 6.2.2]{Roth}), so suppose $\mathcal N$ is a small model of ${T_\forall}$.  By the the same cited lemma, there exists $\mathcal M'\models T$ such that $\mathcal N\into \mathcal M'$. By \Prop{P:LSGU} there is then some small $\mathcal M_0\models T$ with $\mathcal N\into \mathcal M_0$.  By universality, we also have an embedding $\mathcal M_0\into \mathcal M$, proving that $\mathcal N\into \mathcal M$.
%
%
%If , it embeds, by universal preservation (see \cite[Remark on p.~72]{Roth}), in some model of $T$. By the  L\"owenheim-Skolem Theorem (reformulated for Grothendieck universes), it already does in a small model $\mathcal M'\models T$. Universality then yields  an embedding $\mathcal M'\sub \mathcal M$. 
\end{proof} 

\begin{example}\label{E:sets}
Let $\emptyset$ denote the empty theory in the empty language, that is to say, the theory of sets, so that $\md \emptyset$ is just $\catset$. Let $SET^\infty$ (respectively, $SET^n$) be  the theory of infinite sets (respectively, the theory of sets of cardinality $n$). Then $\md{SET^\infty}$ does not contain a terminal object while $\catset$ does, so their complete $\langcat$-theories are not the same.\footnote{None of the $\md{SET^n}$ are elementarily equivalent either, for any hom-set in $\md{SET^n}$ has exactly $n^n$ elements (which is a first-order expressible property).} 
However, the  universal theories are the same  since there is an embedding $\catset\into \md{T^\infty}$ by sending a set $X\in\catset$ to $F^2(X)$, where $F(-):=\hom\catset -\omega$ (since $F(\emptyset)$ is a singleton, we need to apply $F$ twice to always get an infinite set and to maintain covariance). 
%Note that this embedding is even reflective, with reflector $Y\mapsto Y\setminus \omega$. NOT TRUE, $\omega$ would be reflected back to the empty set.
\end{example}

\begin{theorem}\label{T:univcat}
Given a language $L$ and an $L$-theory $T$, the categories $\md T$ and  $\mds T$ are  universal (in the language $\langcat$ of categories). 
\end{theorem} 
\begin{proof}
Let $\mathbb C$ be a small category that is $\langcat$-elementary equivalent with $\md T$. By \Rem{R:ulproof}, there is some ultrapower $\ul{\md T}$  with respect to an ultrafilter on a set $X\in \catset$,  such that $\mathbb C$ embeds elementarily in $\ul{\md T}$ and by \Cor{C:univstrucGU}, the latter embeds in $\md T$. The case of $\mds T$ follows by the same argument, using instead \Cor{C:univstrucGUs}.
\end{proof} 
 
%\begin{corollary}\label{C:genpreserve}
%If $\md T$ admits a generator, then any small $\cat C$ that is elementarily equivalent with  $\md T$, admits an embedding $i\colon \cat C\to \md T$   preserving limits.
%\end{corollary}  
% \begin{proof}
%Let  $\mathcal I$ be a   generator of $\md T$ (which may  not be the one-variable term algebra). In the proof of \Cor{C:univstrucGU}, we can use this generator instead of the term algebra, to get a representable functor $\hom{}{\mathcal I}-$, yielding an embedding $\ul{\md T}\into \md T$.  Being representable, this embedding  preserves limits. As elementary embeddings also must preserve limits, the result follows. 
%\end{proof}  
%

\begin{corollary}\label{C:nonfouniv}
Let $L$ be a first-order language and $E$ a collection of $L$-structures in $\catset$ closed under taking ultraproducts and isomorphic copies. Let $\cat E$ be the full subcategory  of $\md L$ the objects of which are the structures in $E$. Then $\cat E$  is universal.
\end{corollary} 
\begin{proof}
Let $\cat C$ be a small category that is elementarily equivalent with $\cat E$.  By \Rem{R:ulproof} again, there is an ultrafilter in $\catset$ such that $\cat C$ admits an elementary embedding into the  ultrapower $\ul{\cat E}$. Let $i\colon \ul{\md L}\to \md L$ be the embedding given by \Cor{C:univstrucGU} and consider its restriction to the subcategory $\ul{\cat E}$. Let $B$ be an object in $\ul{\cat E}$, given by a sequence $x\mapsto \mathcal B_x$ with $\mathcal B_x\in E$. By assumption, the ultraproduct $\ul{\mathcal B}$ of the latter also lies in $E$. As $i(B)\iso \ul{\mathcal B}$, also $i(B)\in E$, and therefore we get an induced embedding $i\colon \ul{\cat E}\into \cat E$, which composed with the elementary embedding $\mathbb C\into\ul{\cat E}$ yields the desired embedding.
\end{proof} 

\begin{example}\label{E:nonfo}
This yields plenty of examples of universal categories that do not arise as the model category of a first-order theory. For instance, let $S\sub \omega$ be a subset consisting besides $0$ of primes and let $\cat {Fld}_S$ be the subcategory of $\md{FLD}$ of all (small) fields  whose \ch\ belongs to $S$.  As the objects in $\cat {Fld}_S$ are  closed under ultraproducts,   the category   is universal by \Cor{C:nonfouniv}. However, if $S$ is an infinite and co-infinite set, $\cat{Fld}_S$ is not induced by a first-order theory.
% Moreover, the $\langcat$-theory $\mathbf{Fld}_S$ of $\mathbbm {Fld}_S$ can recover the subset $S$ as follows: let $\mathbb C\models\mathbf{Fld}_S$ be in $\catset$. By universality, $\mathbb C$ embeds in $\mathbbm {Fld}_S$. Let $P$ be the collection of all \ch{s} of objects  $B$ in $\mathbb C$ (viewed as a field). Clearly, $P\sub S$. Hence $\mathbb C$  embeds already in $\mathbbm  {Fld}_P$, and hence is a model 
%of $(\mathbf{Fld}_P)_\forall$. Therefore, so is $\mathbbm {Fld}_S$.   Choose for each element in $S$, a field in $\catset$ of that \ch. By \Prop{P:LSGU}, there is an elementary embedding $\mathbb D\into \mathbbm{Fld}_S$ containing this collection of fields with $\mathbb D\in\catset$. Since $\mathbb D$ then also is a model of $(\mathbf{Fld}_P)_\forall$,  we get an embedding $\mathbb D\into \mathbbm{Fld}_P$ by \Prop{P:univth}, proving that $S\sub P$. 
\end{example}

\section{Universalizing theories}\label{s:fo}
%To avoid pathologies, we will assume that structures are never empty. 
To any   theory $T$ in a first-order language, we can now associate a (complete)  $\langcat$-theory $\mathbf T$, namely, the theory of $\md T$. However, if $S$ is an extension of $T$ (e.g., a completion), there is no immediate connection between $\mathbf T$ and $\mathbf S$ (as both are complete theories). It is of course still possible that $\mathbf S=\mathbf T$, and although I do not know any examples, disproving equality is not obvious, and at present, I only know ad hoc arguments.

\begin{example}\label{E:ACFCB}
Let $FLD$ and $ACF$ stand respectively for the theory of fields and the theory of \acf{s} in the language of rings, and let $\mathbf{FLD}$ and $\mathbf {ACF}$ be the corresponding $\langcat$-theories of $\md{FLD}$ and $\md{ACF}$. To see that these are not the same, observe that the Cantor-Bernstein property holds for the latter but not for the former. More precisely, let  $\sym{CB}$ be the sentence expressing that $K\into L$ (i.e., there is a monomorphism $K\to L$) and $L\into K$ implies $K\iso L$; then    $\md{ACF}\models\sym{CB}$ since $K$ and $L$ must then have the same \ch\ and the same transcendence degree, whence are isomorphic by Steinitz's theorem. On the other hand,      $K$  and $K(x)$, with $K$ uncountable and algebraically closed and $x$ a variable, are mutually embeddable
%\footnote{Just take an endomorphism $\varphi\colon K\into K$ that is not an automorphism, so that there must be some transcendental $\alpha\in K$ not in $\op{Im}(\varphi)$, and we can extend $\varphi$ to an embedding $K(x)\into K$, by $x\mapsto \alpha$.} 
but not isomorphic as the latter is not algebraically closed, and hence $\sym{CB}$ does not belong to $\mathbf{FLD}$. 
%A more structural distinction is  the following:   let $\sym{deg}_d(x)$ be the  sentence expressing that for $x\colon K\to L$, there are  exactly $d$ morphisms $j\colon L\to L$ such that $j\after x=x$ (in other words,  the Galois group $\op{Gal}(L/K)$ has order   $d$). Then  $\forall x\niet\sym{deg}_d(x)\in \mathbf{ACF}$ for any $d\geq 2$ and any $x$. 
%Moreover, $\sym{deg}_1(x)$ holds \iff\ $x$ is an isomorphism.

%
%
%let $\sym{iso}(x)$ be the sentence stating that $x$ is an isomorphism and let $\sym{deg}_d(x)$ be the first-order $\langcat$-sentence expressing that for $x\colon K\to L$, there are  at most $d$ morphisms $j\colon L\to L$ such that $j\after x=x$ . Obviously, this cannot happen for \acf{s}, so that   $(\forall x) \sym{iso}(x) \of \niet\sym{deg}_d(x)$, for $d\in\nat$, belongs to $\textbf{ACF}$, but not to $\textbf{FLD}$.

%However, the above sentences  also hold in $\md{RCF}$, where $RCF$ is the theory  of real closed fields (in the language of rings). 
%It too satisfies $\sym{CB}$: if $R\into R'\into R$ are inclusions of real closed fields, and $\bar R$ and $\bar R'$ are their algebraic closures, then $\bar R\into \bar R'\into \bar R$ implies that $\bar R\iso \bar R'$. But since the index of a real closed field inside its algebraic closure is always two, the inclusions $R'\into R\into \bar R\iso\bar R'$ yield  that the first inclusion must be an isomorphism. 
One way, to tell $\md{ACF}$ and $\md{RCF}$ apart: only the latter has an object that embeds in any other object (the field of algebraic reals). Let $ACF_p$ be all \acf{s}\ of \ch\ $p$, which now also has this property, but still the theories $\mathbf{ACF}_p$ and $\mathbf{RCF}$ are distinguishable as only the former is totally ordered: for any two \acf{s} $K$ and $L$ of the same \ch, either $K\into L$ or $L\into K$ by categoricity. This fails miserably in $\md{RCF}$.

%\comment{False: there are function field extensions with trivial automorphism group}
%Note that these degree-sentences can be used to ind-define $\md {ACF}$ as well as  $\md{RCF}$ within $\md{FLD}$: a field $K$ is algebraically closed \iff\  $\sym{deg}_d(f)$ fails for for any  morphism $f$ with domain $K$  and any $d\geq 2$. If instead we only require this for   $d\geq 3$, then we have ind-defined instead the subcategory $\md{RCF}$.
%
%by the following sentence $\sigma$: there exists $K_1\varsubsetneq K_3$ such that for any $L,L'$ with $K_1\varsubsetneq L,L'\varsubsetneq K_3$, we must have $L\iso L'$. To see that this holds in $\md{ACF_p}$ take $K_1$ and $K_3$ to be \acf{s} of cardinality $\aleph_1$ and $\aleph_3$ respectively, so that  up to isomorphism, the only \acf\ in between is that of cardinality $\aleph_2$. 
\end{example}

\begin{remark}\label{R:eecat}
Since $\langcat$ is a countable language, there are at most continuum many complete $\langcat$-theories. On the other hand, there are many more first-order theories in all possible languages in $\catset$. In particular, their will be disjoint languages $L$ and $L'$ and theories $T$ and $T'$ in these respective languages such that $\mathbf{T}=\mathbf {T'}$. Universality now yields a weak form of  definitional equivalence between models of $T$ and models of $T'$ in the following sense. Let $\cat C$ be a small, elementary substructure of $\md T$ (by \Prop{P:LSGU} we can even assure it to contain a given small collection of models). Since $\cat C$ is then also elementarily equivalent to $\md{T'}$, it embeds in it by universality, yielding for each model of $T$ belonging to $\cat C$ a corresponding model of $T'$.
\end{remark} 

\Prop{P:univth} shows that as far as universality is concerned, the universal theory    suffices. But neither the complete nor the universal theory are easily described in concrete cases. The next definition generalizes this in the hope that we can find some concrete axiomatization, like the case of Abelian groups, to obtain universality.

\begin{definition}\label{D:univerth}
Given a language $L\in\catset$, let us call an $L$-theory $U$    \emph{universalizing}, if  admits a model $\mathcal M$ with $D(\mathcal M)\sub\catset$, such that  any small model  $\mathcal N\models U$ embeds in $\mathcal M$.
\end{definition}

It follows that the above $\mathcal M$ is universal, for any small model elementarily equivalent to it, is also a model of $U$, whence embeds in $\mathcal M$.  We express this by saying that $U$ is \emph{universalizing for $\mathcal M$}.  
Clearly the complete   theory $T$ of a universal structure $\mathcal M$ is universalizing for that structure,  as is $T_\forall$, by \Prop{P:univth}.

\Examp{E:sets} shows that the same theory can be universalizing for different structures. The following result shows how any two universal structures then must at least compare. Here, given a theory $T$ in a language $L$, we will say that an $L$-sentence $\sigma$ belongs to  $\Delta_2(T)$, if their   exists a $\forall\exists$-sentence $\sigma_1$ and a $\exists\forall$-sentence $\sigma_2$, such that $\mathcal M\models \sigma\asa\sigma_1\asa\sigma_2$ in each model $\mathcal M$ of $T$.

\begin{proposition}\label{P:univa2}
Suppose an $L$-theory $U$ is universalizing for two $L$-structures $\mathcal M$ and $\mathcal N$. Then $\mathcal M$ and $\mathcal N$ agree on all $\Delta_2(U)$-sentences.
\end{proposition} 
\begin{proof}
We will construct a sequence $\mathcal M_0\into \mathcal N_0\into\mathcal M_1\into\mathcal N_1\into \dots$, with all $\mathcal M_i$ elementary substructures of $\mathcal M$ and all   $\mathcal N_i$ elementary substructures of $\mathcal N$. Assume we have such a chain,  let $\mathcal L$ be their filtered co-limit (i.e., direct limit). Note that $\mathcal L$ is then also the filtered co-limit of the sub-sequences $\mathcal M_0\into\mathcal M_1\dots$ and $\mathcal N_0\into\mathcal N_1\into\dots$ (these chains, however, need not be elementary). Let $\sigma$ be a $\Delta_2(U)$-sentence in $L$ which holds in $\mathcal M$. Since the $\mathcal M_i$ are elementary substructures, $\sigma$ holds in  each of them, whence also in $\mathcal L$ by \cite[Theorem 2.4.6]{Hod}, since $\sigma$ is $U$-equivalent to a $\forall\exists$-sentence. Suppose $\niet\sigma$   holds in $\mathcal N$. Being also in $\Delta_2(U)$,   the same argument shows, via the sub-sequence $\mathcal N_i$,   that $\niet\sigma$ holds in $\mathcal L$, contradiction.

To construct the given sequence, let $\mathcal M_0$ be any small elementary substructure of $\mathcal M$. Since it is then a model of $U$ and $\mathcal N$ is universalizing for $U$, there is an embedding $\mathcal M_0\into\mathcal N$. By \Prop{P:LSGU}, we can find a small elementary substructure  $\mathcal N_0$ of $\mathcal N$ and an embedding $\mathcal M_0\into\mathcal N_0$. We can now repeat this argument, with $\mathcal N_0$ in stead of $\mathcal M_0$, to obtain an embedding $\mathcal N_0\into\mathcal M_1$ with $\mathcal M_1$ a small elementary substructure of $\mathcal M$, etc.
\end{proof}

Let us revisit the theory $\mathbf{AB}$ introduced in the introduction: it consist of the  sentences $\sym{prod}$, $\sym{null}$ and $\sym{gen}$ expressing respectively the existence of  (finite) products, of a null-object, and of a generator, and then the two sentences (Ab1) and (Ab2) asserting respectively the existence of a unique group arrow and that every morphism is compatible with these group laws. 

\begin{proposition}\label{P:Abuniv}
The theory   $\mathbf{AB}$ is universalizing (for $\md{T_{\text{Ab}}}$).
%Let $\cat C$ be a category in $\catset$. If $\cat C$ is a model of  $\mathbf{AB}$, then it  embeds in  $\md {T_{\text{Ab}}}$.  
As a partial converse, if $\cat C$ is a small, full subcategory of $\md{T_{\text{Ab}}}$  admitting products and a null-object, then    $\cat C\models \mathbf{AB}$.
\end{proposition}
\begin{proof}
Let $\cat C$ be a small model of $\mathbf{AB}$, so that in particular, it admits a generator   $I$. Put $\set G:=\hom{\cat C}IG$, for $G$ any object in $\cat C$. Using the unique group arrow $\mu_G\colon G\times G\to G$ on $G$, we can define an addition on $\set G$ as follows: for $\alpha,\beta\colon I\to G$, let $\alpha+\beta:=\mu_G\after(\alpha,\beta)$, where $(\alpha,\beta)\colon I\to G\times G$ is the unique morphism given by the product. It is now not hard to show that the properties of a group arrow  make $\set G$ into a group and any morphism $G\to H$ in $\cat C$ yields a group \homo\ $\set G\to \set H$. Since $I$ is a generator, the functor $\set-$ therefore yields an embedding  into not just $\catset$ but already into $\md {T_{\text{Ab}}}$.

For the second statement, assume now that $i\colon \cat C\into \md{T_{\text{Ab}}}$ is a full embedding and $\cat C$ satisfies $\sym{prod}$ and  $\sym{null}$. Let $G$ be an object in $\cat C$, so that $i(G)$ is by assumption an Abelian group with addition $+$. Since embedding preserve cones, $i(G\times G)$ is a cone in $\md{T_{\text{Ab}}}$ and hence there is a unique morphism $i(G\times G)\to i(G)\times i(G)$ (we do not know whether this is an isomorphism as $i$ might not preserve limits). Composition with $+\colon i(G)\times i(G)\to i(G)$ yields a morphism $i(G\times G)\to i(G)$. Since $i$ is full, there must be a morphism $\mu\colon G\times G\to G$ in $\cat C$ such that $i(\mu)$ is equal to this composition. It is now not hard to see that $\mu$ is a group arrow on $G$ and that axioms (Ab1) and (Ab2) hold in $\cat C$, showing that it is a model of $\mathbf{AB}$.
\end{proof}  

%The first assertion yields the universality of $\md {T_{\text{Ab}}}$, since it itself is a model of $\mathbf {AB}$ and hence so is any category elementary equivalent to it. In fact, we get $\cateo{T_{\text{Ab}}}\sub \mathbf{AB}$. It is not clear whether we can drop the condition that the subcategory is full in the second statement. If we can omit it, we would have proven that $\mathbf{AB}$ is equivalent with $\cateo{T_{\text{Ab}}}$ together with  the three sentences  $\sym{prod}$, $\sym{null}$ and $\sym{gen}$. 
%%Note that our axiomatization of $\mathbf{AB}$ is definitely not universal, so we cannot expect both theories to be just equal. 

\section{Weakly universal categories}

In this section, we will construct some ``algebraic'' categories that have some universal behavior but are not given as the models of a first-order theory. 
%More precisely, these are \conc\ categories $\mathbb C$ that are not \qfo. 
%closed under ultraproducts (or, more precisely, their underlying sets aren't). 

\begin{definition}\label{D:wuniv}
Let us  call a \conc\ category $\mathbb C$ \emph{weakly universal}, if for any small category $\mathbb D$ elementarily equivalent to $\mathbb C$, there is a functor $i\colon \mathbb D\to \mathbb C$ which is injective on objects.
\end{definition} 

In other words, $i$ might no longer  be faithful. 
The following example already exhibits our main technique: the use of cataproducts. Recall that if $(R_x,\maxim_x)$ are Noetherian local rings, then their \emph{cataproduct} $\cp R$ is by definition the quotient of their ultraproduct $(\ul R,\ul\maxim)$ by the \emph{ideal of infinitesimals}, that is to say, the intersection of all powers $\ul\maxim^n$. By \cite[Thm 8.1.4]{SchUlBook}, if there is a bound on the embedding dimension of the $R_x$, that is to say, if there is a bound on the minimal number of generators of the $\maxim_x$, then $\cp R$ is again Noetherian (of the same embedding dimension as almost all $R_x$) and, moreover, complete.

\begin{example}\label{E:nonfo}
Consider the category $\cat{Dvr}$  of \DVR{s} of equal \ch\ zero, viewed as a full subcategory of the category of local rings. It has as generator $P:=\pol {\mathbb Q}t_{(t)}$. The corresponding forgetful functor $\set-$ then takes a \DVR\ $V$ to its maximal ideal $\maxim\iso\hom{\cat{Dvr}}PV$. However, we can always recover the \DVR\ $V$ from the ring operations on its maximal ideal $\maxim$: 
\begin{description}
\item[PI] there exists an element $t\in \maxim$ such that for any $a,b\in\maxim$, there is a unique $z\in\maxim$ such that $ab=tz$.
\end{description}
 Setting $a*b:=z$ then yields a new multiplication on $\maxim$ which gives it a ring structure isomorphic to $V$.  Therefore, the associated set $\set A$ of an object  $A$ in $\ul{\cat{Dvr}}$ given by a sequence of \DVR{s} $A_x$ is by    \Lem{L:isoulGU} the ultraproduct of  the $\maxim_x$, and hence in particular satisfies again  condition (PI). The resulting ring $\ul A$ is now easily seen to be the ultraproduct of the $A_x$. Unfortunately, it will no longer belong to 
 $\cat{Dvr}$ (as its value group is the ring of non-standard integers). However, its cataproduct $\cp A$ does. Moreover, if $B$ is an object different from $A$, then $\set A$ and $\set B$ are disjoint,  and so therefore are their quotient spaces $\cp A$ and $\cp B$. It is now not hard to see that the assignment $\ul{\cat{Dvr}}\into \cat{Dvr}$ sending an object $A$ to the cataproduct $\cp A$ is functorial and injective on objects. As  in the proof of \Thm{T:univcat}, it then follows that $\cat{Dvr}$ is weakly universal. Note however, that the notion of being a \DVR\ is not first-order definable, since the class is not closed under ultraproducts, only under cataproducts.
 
 The restriction to equal \ch\ zero is just to fix the choice of generator. In \ch\ $p$, we should instead take instead $P:=\pol {\mathbb F_P}t_{(t)}$. However, there does not seem to be such a generator in mixed \ch. But in equal \ch, we now have a family of generators and they are actually locally unique and the same argument proves then that this larger category of \DVR{s} is weakly universal.
\end{example}

The key result we will use, of which the above is also an ad hoc application, is the following (the proof is deferred to a future paper).

\begin{proposition}\label{P:rngnormal}
Let $R$ be a  normal Noetherian local ring of positive dimension and view its maximal ideal $\maxim$ as a structure in the language of rings. Then $R$ is interpretable in $\maxim$. More precisely, there  are formulas in the language of rings $\rho(x)$ and $\mu(x,y,z)$ such that the subset defined by $\mu$ in $\maxim^3$ is the graph of a function $*$  on the subset $Z\sub \maxim$ defined by $\rho$ and such that $(Z,+,*)$  is a ring which is isomorphic to $R$.
\qed
\end{proposition} 
\begin{remark}\label{R:rngnormal}
The $Z\sub\maxim$  defined by $\rho$ is in fact an invertible ideal in $R$, and it is from this invertible ideal that one can `reconstruct' $R$. 
\end{remark} 

\begin{theorem}\label{T:univnormal}
The category $\cat{Norm}_e$ of small, positive dimensional, equal \ch, normal, Noetherian local domains of embedding dimension at most $e$  is weakly universal.
\end{theorem} 
\begin{proof}
Let $\Theta$ be the collection of all $\pol kt_{(t)}$, where $t$ is a single variable and $k$ runs over all prime fields, that is to say, $k$ is either $\mathbb Q$ or a $p$-element field $\mathbb F_p$. It is not hard to see that these form a family of generators of $\cat{Norm}_e$ and that the family is locally unique. For $R$ in $\cat{Norm}_e$, we set
\[
\set R:= \hom{\cat{Norm}_e}{\Theta}R.
\]
Since in $\cat{Norm}_e$, the morphisms are assumed to be local, $\set R$ is isomorphic to the maximal ideal of $R$. Take an ultrapower $\ul{(\cat{Norm}_e)}$ with respect to an ultrafilter on a set $X\in \catset$ and let $\ul{\Theta}$ be the image of the family $\Theta$ under the diagonal embedding, so that they form a locally unique family of generators of the ultrapower $\ul{(\cat{Norm}_e)}$.  For an object $A$ in the ultrapower given by a sequence  of Noetherian local normal domains $A_x$ of embedding dimension at most $e$, we therefore set again $\set A:=\hom{\ul{(\cat{Norm}_e)}}{\ul{\Theta}}A$,  and then by \Lem{L:isoulfam},  this set  $\set A$  is just the ultraproduct of the sets $\set{A_x}$. Using the notation from \Prop{P:rngnormal}, the subset $Z\sub \set A$ defined by $\rho$ and the function $*$ on $Z$ defined by $\mu$ are the respective ultraproducts of the subsets $Z_x$ and the binary functions $*_x$ for each $x$. Since $(Z_x,+,*_x)\iso A_x$, it follows that $(Z,+,*)$ is isomorphic to the ultraproduct of the $A_x$, whence is in particular a local ring with maximal ideal $\maxim$ generated by at most $e$ elements. Let $\cp A$ be the cataproduct of the $A_x$, which therefore is isomorphic to $Z/\mathfrak M$ where $\mathfrak M$ is the intersection of all powers $\maxim^n$. By \cite[Them 8.1.4]{SchUlBook}, the cataproduct $\cp A$ is a complete Noetherian local domain of embedding dimension at most $e$. Let $\cp{\bar A}$ be the integral closure of $\cp A$,\footnote{It is very likely that $\cp A$ is already normal, but since I do not have immediately a proof, I just add this step.} so that it belongs to $\cat{Norm}_e$. We have defined an assignment  $\ul{(\cat{Norm}_e)}\to \cat{Norm}_e\colon A\mapsto \cp {\bar A}$ and it is now easy to check that this map is functorial.  The usual argument using \Prop{P:KSirr} then yields weak universality. 
\end{proof}

Unfortunately, while  injective on objects, the functor  $A\mapsto \cp {\bar A}$ is no longer so on morphisms. For instance the endomorphism on $\ul{(\pol kt_{(t)})}$ given by the sequence of morphisms  $t\mapsto t^n$ (with index set $X=\nat$) is different in $\ul{(\pol kt_{(t)})}$ from the endomorphism given by $t\mapsto 0$, but in the cataproduct, they are the same.

%We discuss its injectivity in the remark below, so that we have shown that $\ul{(\mathbb N\text{orm}_e)}$ is a subcategory of $ \mathbb N\text{orm}_e$. Universality now follows by the same argument as in \Thm{T:univcat}.
%
%\begin{remark}\label{R:univnormal}
%Let us write by $R^*:=R\setminus \{0\}$ for a domain $R$. On $R\times R^*$, we define an equivalence relation $\sim$ by $(a,b)\sim (c,d)$ \iff\ $ad=bc$. Hence $R\times R^*/\sim$ is a representation of the field of fractions of $R$. In particular, the integral closure $\bar R$ of $R$ is the quotient of some (definable) subset of $R\times R^*$. Let $A$ and $B$ now be distinct objects in $\ul{(\mathbb N\text{orm}_e)}$, so that $\set A$ and $\set B$ are disjoint sets, and hence so are the respective quotients giving $\cp{\bar A}$ and $\cp{\bar B}$, showing that the above functor is injective on objects. It is faithful, since $\mathbf I$ is a family of generators, and taking cata-products as well as normalizations are also faithful functors. 
%\end{remark}  

\begin{corollary}\label{C:univnonfo}
For $d,e$ positive integers, the following categories of small $d$-dimensional equal \ch\ normal local rings of multiplicity $e$ are weakly universal:    \CM,  Gorenstein, complete intersection. In particular, the category of equal \ch\ $d$-dimensional regular local rings is weakly universal. 
\end{corollary} 
\begin{proof}
The last assertion  is just the case $e=1$. All given categories are subcategories of $\cat{Norm}_{d+e-1}$  and the only thing to observe now is that the stated properties are preserved under cataproducts. For regularity, this follows from \cite[Corollary 8.2.]{SchFinEmb},    for the \CM\ condition from \cite[Theorem 8.8.]{SchFinEmb}, and the remaining two are dealt with by  \cite[Theorem 8.12.]{SchFinEmb}.
\end{proof} 
\begin{remark}\label{R:univnonfo}
Since cataproducts are always complete, universality also holds for the subcategories of the above categories consisting only of the complete local rings in that category.
\end{remark}

\section{Appendix: equivalence of categories versus elementary equivalence.}\label{s:app}

\newcommand\val[3]{\sym{QC}(#1,#2;#3)}
\newcommand\vali[4]{\sym{QC}_{\cat #4}(#1,#2,#3)}
\newcommand\isoi{\iso_{\cat i}}
\newcommand\langhomot{L_{\text{homo}}}
\newcommand\homot{homotopic}
\newcommand\ito{\Rightarrow}

Let $\cat C$ be a category. By an \emph{iso-graph} of $\cat C$, we mean a thin,  wide\footnote{\emph{Wide} means that they have the same objects; \emph{thin} means  that the    hom-sets have at most one element.} subcategory $\cat i$  such that  $\hom{\cat i}AB\neq\emptyset$ \iff\ $A\iso B$ (in $\cat C$). In particular, if $\cat C$ is skeletal, then $\cat i$ is discrete (whence unique), but for non-skeletal categories, there may be many iso-graphs (assuming the axiom of choice, iso-graphs always exist.)

Note that if $\alpha\in \hom{\cat i}AB$, then it is actually an isomorphism: since necessarily $A\iso B$, there is $\beta\in \hom{\cat i}BA$ and hence $\beta\alpha\in \hom{\cat i}AA=\{1_A\}$, so that in fact, $\beta=\inv\alpha$. Therefore, $\cat i$ is an iso-graph, if it is a subcategory with the same objects in which all morphisms are isomorphisms and the identities are the only automorphisms (this forces the category to be thin as well). Let us indicate an arrow belonging to $\cat i$ by $A\ito B$, whenever the iso-graph is clear. Given morphisms  $f\colon A\to B$, $g\colon C\to D$ and $h\colon P\to Q$, let  $\vali fghi$ express the existence of the following commutative diagram (of \emph{quasi-composition})
\begin{equation}\label{eq:Ival}
\xymatrix{
P\ar@{=>}[d]\ar[rrr]^h&&&Q\ar@{=>}[d]\\
A\ar[r]_f&B\ar@{=>}[r]&C\ar[r]_g&D.
}
\end{equation} 
Of course, implicit in this is that $A\iso P$, $B\iso C$ and $D\iso Q$. 

\begin{lemma}\label{L:comp}
Given morphisms $f\colon A\to B$, $g\colon B\to D$ and $h\colon A\to D$, then $\vali fghi$ \iff\ $h=g\after f$.
\end{lemma} 
\begin{proof}
One direction is immediate, so assume $\vali fghi$ holds. But then in  diagram~\eqref{eq:Ival}, all $\ito$-arrows must be identities, proving that $h=gf$.
\end{proof} 

\begin{lemma}\label{L:Iff}
Let $\cat C$ and $\cat D$ be categories with respective iso-graphs $\cat i$ and $\cat j$. If $F\colon \cat C\to \cat D$ is a fully faithful functor preserving these iso-graphs (i.e., inducing a functor $\cat i\to\cat j$), then $\vali fghi$ holds in $\cat C$, for  some $\cat C$-morphisms $f,g,h$,   \iff\ $\vali{F(f)}{F(g)}{F(h)}j$ holds in $\cat D$.
\end{lemma} 
\begin{proof}
One direction  holds for any functor preserving the iso-graphs. To prove the converse, let us first show that $A\ito B$ \iff\ $F(A)\ito F(B)$. One direction is by assumption, so assume   $\alpha\colon F(A)\ito F(B)$. Since then also $\inv\alpha\colon F(B)\ito F(A)$ and $F$ is full, there are $a\colon A\to B$ and $b\colon B\to A$ such that $F(a)=\alpha$ and $F(b)=\inv\alpha$.  Hence $F(ba)=1_{F(A)}=F(1_A)$ and faithfulness then implies $ba=1_A$. In particular, $A\iso B$ and hence there is $\sigma\colon A\ito B$, and by assumption $F(\sigma)=\alpha$.

 Therefore, if $\vali{F(f)}{F(g)}{F(h)}j$ holds, i.e, if we have a commutative diagram
\[
\xymatrix{
F(P)\ar@{=>}[d]\ar[rrr]^{F(h)}&&&F(Q)\ar@{=>}[d]\\
F(A)\ar[r]_{F(f)}&F(B)\ar@{=>}[r]&F(C)\ar[r]_{F(g)}&F(D).
}
\]
in $\cat D$, then this is the image under $F$ of  diagram \eqref{eq:Ival} and since $F$ is faithful, the latter diagram also commutes, i.e., $\vali fghi$ holds.
\end{proof} 

\begin{lemma}\label{L:subisogr}
Let $\cat d$ be a finite subcategory of $\cat C$. There is an $\langcat$-formula $\sym{subisogr}(\cat d)$ expressing that $\cat d$ can be extended to an iso-graph of $\cat C$.
\end{lemma} 
\begin{proof}
Apart from the obvious constraints on the hom-sets in $\cat d$, we must also check that no composition of morphisms in $\cat d$ can yield a proper automorphism.  If $\cat d$ has $n$ morphisms, then this requires only $n!$ the `words' in these morphisms to be checked.
\end{proof} 

%We cannot alow the predicate $ID$ as it is not stable under equivalence
%Let us also introduce a unary predicate in the $\sym m$-sort: let $\sym{ID}$ be the collection of all identity morphisms; we will identify the elements of $\sym{ID}$ with the objects in  the category. Note that $\cat i$ is then definable with aid of these two predicates: a morphism $f$ belongs to $\cat i$ \iff\ there exists $i,j\in\sym{ID}$  such that $\vali ijfi$ holds.  This is just an example of the following type of   formulas: fix an iso-graph $\cat i$. 
Note that the collection of morphisms of the iso-graph $\cat i $ is  also definable as 
\begin{equation}\label{eq:defi}
 \alpha\ \text{is an $\cat i$-morphism \quad if and only if\quad}(\forall p)  [(\exists q) \vali \alpha pqi\dan \vali \alpha ppi].
\end{equation} 
Indeed, for a given $p$, if there is some $q$ such that $\vali \alpha pqi$ holds, then the range $B$ of $\alpha$ is isomorphic to the domain of $p$. So, if $\alpha$ is moreover an $\cat i$-morphism, then $\vali \alpha ppi$ holds. Conversely, if this holds, then take $p=q=1_B$, so that   the quasi-composition diagram~\eqref{eq:Ival} shows that $\alpha$ belongs to $\cat i$.  This defining formula is an example of the following class of formulas:
%
%We will view a given iso-graph $\cat i$  as a unary predicate (on the $\sym m$-sort).

\begin{definition}\label{D:homot}
Any  expression   built up from the    predicate  $\sym{QC}_{\cat i}$  by taking Boolean combinations and quantification will be called a \emph{\homot} formula. 
\end{definition} 
Note that equality is not allowed in such formulas as functors are in general not injective (even if they are faithful).  Unlike the systems in the literature (like \cite{Blanc,FreydEquiv,MakPar,Shul}) using dependent types, quantification here is unrestricted.

 \begin{theorem}\label{T:equivelem}
% Let $\cat C$ and $\cat D$ be  equivalent of categories and let $\cat i$ and $\cat j$ be respective iso-graphs. If $
% 
%  and let $\cat i$ be a iso-graph of 
Any two equivalent categories agree on  \homot\ sentences.
\end{theorem}
\begin{proof}
We will prove the stronger statement that if $\cat C$ is the skeleton of $\cat D$, then any  \homot\ sentence $\sigma$ with parameters from $\cat C$ holds in $\cat C$ \iff\ it holds in $\cat D$. The general result follows since two categories are equivalent \iff\ they have isomorphic skeletons. 

Note that $\cat C$ has a unique iso-graph, namely the underlying discrete category.  Fix some iso-graph $\cat i$ in $\cat D$. I claim there exists an equivalence $G\colon \cat D\to \cat C$ which preserves iso-graphs and is the identity on $\cat C$. Indeed, given an arbitrary object $A$ in $\cat D$,  we let $G(A)$ be the unique object in $\cat C$ that is isomorphic with $A$. Given a morphism $f\colon A\to B$, let $\alpha\colon G(A)\ito A$ and $\beta\colon G(B)\ito B$  be the unique $\cat i$-morphisms and now set $G(f):=\inv\beta f\alpha$.  One  checks that $G$ is a full, faithful and (essentially) surjective functor which sends morphisms in $\cat i$ to identities and whose restriction to $\cat C$ is the identity functor.

  If $\sigma$  has no quantifiers, then the result follows from \Lem{L:Iff} applied to $G$. By induction on the number of quantifiers, we may reduce to the case that $\sigma$ is of the form $(\exists x)\varphi(x)$, where $\varphi(x)$ is a \homot\ formula such that for each $a$ in $\cat C$, both categories agree on the validity of $\varphi(a)$, and we have to show that the same is true for $\sigma$.  The non-trivial direction is that there is $b$ in $\cat D$, such that $\cat D\models \varphi(b)$.   Since the only parameter in $\varphi(b)$ not belonging to $\cat C$ is $b$, \Lem{L:Iff}   yields $\cat C\models \varphi(G(b))$, and hence $\sigma$ also holds in $\cat C$.
\end{proof}  

\subsection*{Limits}
Our goal is to show that the existence of finite limits (and co-limits) can be expressed by \homot\ sentences. 
To facilitate our discussion, let us introduce the following notations for   morphisms $p\colon A\to B$ and $q\colon C\to D$ in a category $\cat C$ with a fixed iso-graph $\cat i$. Write  $p\isoi q$ to mean the existence of a commutative diagram  
\begin{equation}\label{eq:isoeq}
\xymatrix{
A\ar[r]^p\ar@{=>}[d]&B\ar@{=>}[d]\\
C\ar[r]_q&D.
}
\end{equation}
We indicate this by  writing $p\colon C\leadsto_{\cat i} D$, without mentioning $q$, and call $p$ a \emph{quasi-$\cat i$-morphism} from $C$ to $D$. Clearly, $\isoi$ is an equivalence relation. When $\cat i$ is clear from context, we  omit it,   just writing $p\iso q$, $C\leadsto D$, etc.  Note that $p\iso q$ \iff\ there exists an $\cat i$-morphism $\alpha$ such that $\val \alpha pq$, and hence using  \eqref{eq:defi},  we get   a \homot\ formula defining $\iso$.
%
%, so that  $p\iso q$ is equivalent with the \homot\ formula $(\exists x) (\forall y) \val xyy \en \val xpq$.
%
%
%there exists some $i\in\sym{ID}$ such that $\vali ipqi$ (in fact, it suffices for $i$ to belong to $\cat i$). In other words, there is  a commutative diagram (with $i=1_A$)
%\begin{equation}\label{eq:isoeq}
%\xymatrix{
%A\ar[r]^p\ar@{=>}[d]&B\ar@{=>}[d]\\
%C\ar[r]_q&D.
%}
%\end{equation}
%In that case, we will also just write $p\colon C\leadsto_{\cat i} D$, without mentioning $q$; we call $p$ a \emph{quasi-$\cat i$-morphism} from $C$ to $D$. When $\cat i$ is clear from context, we just omit it   just write $p\iso q$, $C\leadsto D$, etc. 
%% (of course, we could also have written this as $q\colon A\leadsto B$).
%\begin{lemma}\label{L:isoeq}
 Moreover, by \Lem{L:comp}, we have  
 \begin{equation}\label{eq:isoeq}
p\iso q \asa p=q \quad \text{for all parallel morphisms}\  p,q\colon A\to B. 
\end{equation} 

%\end{lemma}
%\begin{proof}
%In \eqref{eq:isoeq}, we have $A=C$ and $B=D$ so that the vertical  arrows are identities.
%\end{proof}  

Fix a finite category $\cat d$ and let $J\colon\cat d\to \cat C$ be a functor (viewed as a parameter encoding the finitely many $J(A)$ and $J(s)$ with $A$ and $s$ running over all respective objects and morphisms of $\cat d$); we refer to this situation by calling $J$ a \emph{finite $\cat d$-diagram} in $\cat C$. Define a \homot\ formula $\sym{qcone}_J(x,y)$ (a \emph{quasi-$J$-cone}) with $x\colon C\ito C'$ being some  morphism in $\cat i$ and the $y$-variables corresponding to  quasi-morphisms $\psi_A\colon C\leadsto J(A)$, for $A$  running over all objects in $\cat d$, subject to the requirement that $\val {\psi_A}{J(s)}{\psi_B}$, where $s\colon A\to B$ runs over all morphisms in $\cat d$. In particular, there are morphisms $\varphi_A\colon C\to J(A)$ such that $\psi_A\iso \varphi_A$, for each $A$. \Lem{L:comp} then yields that  $(C,\varphi_A)$ is an actual $J$-cone. 

Next, we define the formula $\sym{qlim}_J(x)$, where $x\colon L\ito L'$ is again some $\cat i$-morphism, expressing that there exist $\lambda_A$ such that $\sym{qcone}_J(x,\lambda_A)$ holds (i.e., the $\lambda_A$ form a quasi-$J$-cone), with the property that whenever $\sym{qcone}_J(\alpha,\psi_A)$ holds, for some $\alpha\colon C\ito C'$,  there is a quasi-morphism $u \colon C\leadsto L$ such that $\val u{\lambda_A}{\psi_A}$ holds, for all $A$, and, if $v\colon C\leadsto L$ is a second quasi-morphism  such that  $\val v{\lambda_A}{\psi_A}$,   for all $A$, then $u\iso v$.

\begin{proposition}\label{P:lim}
Let $\cat C$ be a category with a fixed iso-graph $\cat i$ and let $J\colon \cat d\to \cat C$ be a finite diagram. Then the \homot\ sentence $(\exists x)\sym{qlim}_J(x)$ holds in $\cat C$ \iff\ $J$ has a limit in $\cat C$. 
\end{proposition} 
\begin{proof}
One way is obvious, so assume $\sym{qlim}_J(\alpha)$ holds, where $\alpha\colon L\ito L'$. In particular, by assumption, there is quasi-$J$-cone $(L,\lambda_A)$, that is to  say, $\sym{qcone}_J(\alpha,\lambda_A)$ holds. We already remarked that we may  assume that it is an actual cone. Let $(C,\varphi_A)$ be an arbitrary cone. Viewed as a quasi-$J$-cone, there is a quasi-morphism $u\colon C\leadsto L$ such that $\val u{\lambda_A}{\varphi_A}$ holds for all $A$. Replacing $u$ by an isomorphic copy, we may assume that $u\colon L\to C$ is an actual morphism, and hence  $\varphi_A=\lambda_A\after u$ by \Lem{L:comp}. We need to show that $u$ is unique with this property. So let $v\colon C\to L$ be a second morphism satisfying $\varphi_A=\lambda_A\after v$.  By $\sym{qlim}_J(\alpha)$, we get $u\iso v$, whence $u=v$ by \eqref{eq:isoeq}. 
\end{proof} 

 Obviously, co-limits can likewise be encoded by \homot\ formulas. Combining all this with \Thm{T:equivelem}, we recover the well-known fact: 
 
\begin{corollary}\label{C:equivlim}
Two equivalent categories have the same finite limits and co-limits.\qed
\end{corollary} 

\subsection*{Conclusion} 
While model-theorist normally consider  the equality sign as  a logical symbol, one could opt not to do so. Consider then the \emph{\homot\ language} $\langhomot$ consisting of the  binary predicate $\iso$ and the  ternary predicate $\sym{QC}$ in which atomic formulas (whence general formulas)  are built in the usual way with the exception that in stead of $=$, we must use $\iso$ (in fact, we can leave out the predicate $\iso$ as it is definable from $\sym{QC}$). In other words,  instead of term equations, we only have term isomorphisms; note also that the only terms are variables (or morphisms when allowing parameters).  Hence the $\langhomot$-formulas are now just the \homot\ formulas (after replacing each occurrence of the $\iso$ symbol by its \homot\ definition). 

Any category $\cat C$ can be made into a $\langhomot$-structure by  choosing some iso-graph $\cat i\sub \cat C$, taking as underlying set the morphisms of $\cat C$  and letting $\iso$ and $\sym{QC}$  be defined as above.  \Thm{T:equivelem} therefore states that in this language elementary equivalence is the same as equivalence of categories. For skeletal categories,  there is no difference:

\begin{proposition}\label{P:skel}
Let $\cat C$ be  a  skeletal category viewed as an $\langhomot$-structure. Then it can be expanded to a  structure in the language $\langcat\cup\langhomot$ such that any $\langcat$-formula is equivalent to a $\langhomot$-formula, and conversely.
\end{proposition}
\begin{proof}
Since $\cat C$ is skeletal, there is only one iso-graph, the underlying discrete category. The collection of morphisms of this iso-graph is definable as all $\alpha$ such that $(\forall y)\val \alpha yy$ holds, and  consists exactly of all identities.  
For each  such morphism $\alpha$, introduce a new element of the $\sym o$-sort $\set\alpha$, while taking the elements of the $\langhomot$-structure as the those of the $\sym m$-sort. 

Since isomorphic objects are equal, the domain and range predicates are   defined by 
\[
\sym{dom}(f)=\set\alpha \asa\ \val \alpha ff\qquad\text{and}\qquad\sym{rng}(f)=\set\beta \asa\ \val f\beta f
\]
while by \Lem{L:comp}, composition is defined by $h=g\after f$ \iff\ $\val fgh$. The unary function $\sym{Id}$ is just $\set\alpha\mapsto \alpha$. Finally, equality is encoded by $\iso$ in view of  \eqref{eq:isoeq}. This defines every $\langcat$-formula (with equality) from the ternary predicate $\sym{QC}$, and \eqref{eq:Ival} conversely defines $\sym{QC}$ in $\langcat$. 
\end{proof}  

In \Thm{T:univcat}, the language $\langcat$ only featured to ensure that an elementarily  equivalent structure is again a category, as the rest of the argument is based on ultraproducts. Therefore, we could as well have viewed categories in a different language.  Let us then call a \conc\ category $\cat C$ \emph{homotopically universal} if any small category $\cat D$ which is $\langhomot$-elementarily equivalent to it, embeds in it. The argument in \Thm{T:univcat} then shows that any category of the form $\md T$, for $T$ some theory in some first-order language, is homotopically universal. Moreover, by \Thm{T:equivelem}, any \conc\ category that is equivalent (in the sense of categories) to $\md T$ is then also homotopically universal.

%\bibliographystyle{amsplain}
%\bibliography{myabbrev,references}
\providecommand{\bysame}{\leavevmode\hbox to3em{\hrulefill}\thinspace}
\providecommand{\MR}{\relax\ifhmode\unskip\space\fi MR }
\providecommand{\MRhref}[2]{%
  \href{http://www.ams.org/mathscinet-getitem?mr=#1}{#2}
}
\providecommand{\href}[2]{#2}

\end{document}